\documentclass{article}

\usepackage[final, nonatbib]{neurips_2022}

\usepackage[utf8]{inputenc} 
\usepackage[T1]{fontenc}    
\usepackage{hyperref}       
\usepackage{url}            
\usepackage{booktabs}       
\usepackage{amsfonts}       
\usepackage{nicefrac}       
\usepackage{microtype}      
\usepackage{xcolor}         

\usepackage{comment}
\usepackage{hyperref}
\usepackage{color}
\usepackage{array,multirow}
\usepackage{array}

\usepackage{amsmath}
\usepackage{mathtools}
\usepackage{amssymb}
\usepackage{amsthm}
\usepackage{comment}
\usepackage{array,multirow}
\usepackage{graphicx}
\usepackage{subcaption}
\usepackage{appendix}
\usepackage{color}
\usepackage[noabbrev,capitalize]{cleveref}

\usepackage{amsmath}
\usepackage{algorithm}
\usepackage{amsthm}
\usepackage{amsfonts}
\usepackage{comment}
\usepackage{graphicx}
\usepackage{hyperref}
\usepackage{color}
\usepackage{graphicx}
\usepackage{subcaption}
\usepackage{color}
\usepackage{array,multirow}

\usepackage{array}
\newcolumntype{P}[1]{>{\centering\arraybackslash}p{#1}}
\usepackage{algpseudocode}

\usepackage{pifont}

\newtheorem{theorem} {Theorem}
\newtheorem{lemma} {Lemma}
\newtheorem{definition} {Definition}

\newtheorem{assumption} {Assumption}

\newtheorem{remark} {Remark}

\def\1{{\mathbf{1}}}

\def\u{{\mathbf{u}}}
\def\v{{\mathbf{v}}}
\def\q{{\mathbf{q}}}
\def\z{{\mathbf{z}}}
\def\w{{\mathbf{w}}}

\def\X{{\mathbf{X}}}
\def\Y{{\mathbf{Y}}}
\def\A{{\mathbf{A}}}
\def\M{{\mathbf{M}}}

\def\I{{\mathbf{I}}}
\def\B{{\mathbf{B}}}

\def\V{{\mathbf{V}}}
\def\Z{{\mathbf{Z}}}
\def\W{{\mathbf{W}}}

\def\Q{{\mathbf{Q}}}

\def\P{{\mathbf{P}}}

\DeclareMathOperator*{\argmin}{arg\,min}
\DeclareMathOperator*{\argmax}{arg\,max}

\newcommand{\R}{\mathbf{R}}
\newcommand{\mV}{\mathcal{V}}
\newcommand{\mF}{\mathcal{F}}

\newcommand{\mP}{\mathcal{P}}

\newcommand{\mS}{\mathcal{S}}

\newcommand{\mbS}{\mathbb{S}}

\newcommand{\trace}{\textrm{Tr}}
\newcommand{\rank}{\textrm{rank}}
\newcommand{\conv}{\textrm{conv}}
\newcommand{\reals}{\mathbb{R}}

\title{Local Linear Convergence of Gradient Methods for  Subspace Optimization via Strict Complementarity}

\author{%
  Ron Fisher\\
Technion - Israel Institute of Technology \\
Haifa, Israel 3200003 \\
\texttt{ronfisher@campus.technion.ac.il}\\
   \And
  Dan Garber \\
Technion - Israel Institute of Technology \\
Haifa, Israel 3200003 \\
\texttt{dangar@technion.ac.il}\\
}

\begin{document}

\maketitle

\begin{abstract}
We consider optimization problems in which the goal is find a $k$-dimensional subspace of $\mathbb{R}^n$, $k<<n$, which minimizes a convex and smooth loss. Such problems generalize the fundamental task of principal component analysis (PCA) to include robust and sparse counterparts, and logistic PCA for binary data, among others. This problem could be approached either via nonconvex gradient methods with highly-efficient iterations, but for which arguing about fast convergence to a global minimizer is difficult or, via a convex relaxation for which arguing about convergence to a global minimizer is straightforward, but the corresponding methods are often inefficient in high dimensions. In this work we bridge these two approaches under a strict complementarity assumption, which in particular implies that the optimal solution to the convex relaxation is unique and is also the optimal solution to the original nonconvex problem. Our main result is a proof that a natural nonconvex gradient method which is \textit{SVD-free} and requires only a single QR-factorization of an $n\times k$ matrix per iteration, converges locally with a linear rate. We also establish linear convergence results for the nonconvex projected gradient method, and the Frank-Wolfe method when applied to the convex relaxation. 
\end{abstract}

\section{Introduction}\label{sec:intro}
We consider the problem of finding a $k$-dimensional subspace of $\mathbb{R}^n$, $k<<n$, which minimizes a given objective function, where we identify a subspace with its corresponding projection matrix. That is, we consider the following optimization problem:
\begin{align}\label{eq:optProbNonconv}
\min_{}f(\X) \quad \textrm{subject to} \quad \X\in\mP_{n,k} := \{\Q\Q^{\top} ~|~\Q\in\reals^{n\times k},~\Q^{\top}\Q=\I\}.
\end{align}
Throughout this work and unless stated otherwise, we assume that $f(\cdot)$ is convex, $\beta$-smooth (gradient Lipschitz) and, for ease of presentation, we also assume that  the gradient $\nabla{}f(\cdot)$ is a symmetric matrix over the space of $n\times n$ symmetric matrices $\mbS^n$\footnote{in case the gradient is not a symmetric matrix at some point $\X\in\mbS^n$, then denoting it by $\nabla_{\textrm{nonsym}}f(\X)$, we can always take its symmetric counterpart $\nabla{}f(\X) = \frac{1}{2}(\nabla_{\textrm{nonsym}}f(\X)+\nabla_{\textrm{nonsym}}f(\X)^{\top})$ and, unless stated otherwise, our derivations throughout this work will remain the same}.

Problems of interest that fall into this model include among others robust counterparts of PCA, which are based on the smooth and convex Huber loss (see concrete examples in Section \ref{sec:experiments}), logistic PCA \cite{landgraf2020}, and sparse PCA \cite{vu2013fantope}. Note that in Problem \eqref{eq:optProbNonconv} we are interested in the low-dimensional subspace itself (as opposed to the projection of the data onto it, as in many other formulations), which is important for instance when the end goal is to perform dimension reduction, which is one of the most important applications of PCA-style methods.

Motivated by high-dimensional problems, we are interested in highly efficient (in particular in terms of the dimension $n$) first-order methods for Problem \eqref{eq:optProbNonconv}. Moreover, we are interested in establishing, at least locally, fast convergence to the global minimizer, despite the fact that Problem \eqref{eq:optProbNonconv} is nonconvex. Subspace recovery/optimization problems similar to Problem \eqref{eq:optProbNonconv} have received significant interest in recent years, see for instance \cite{shamir2016fast, zhang2016global, chen2015fast, lerman2018overview, maunu2019well, hardt2013algorithms, lerman2018fast, raghavendra2020list, rahmani2016randomized} however, different from these works, our approach will not assume that $f(\cdot)$ admits a very specific structure (e.g., a linear or quadratic function), or will be based on a specific underlying statistical model. Instead, we will be interested in deterministic conditions that may hold for quite general $f(\cdot)$ (which is convex and smooth), and may render quite a wide variety of  problems ``well-posed'' for efficient optimization. 

We begin by briefly describing two natural dimension-efficient first-order methods for tackling Problem \eqref{eq:optProbNonconv}. 
One  such method is the nonconvex projected gradient method which follows the dynamics: 
\begin{align}\label{eq:nonconvProjGrad}
\X_{t+1} \gets \Pi_{\mP_{n,k}}[\X_t - \eta_t\nabla{}f(\X_t)],
\end{align} 
where $\Pi_{\mP_{n,k}}[\cdot]$ denotes the Euclidean projection onto the set $\mP_{n,k}$ (note that since this set is nonconvex, in general, the projection need not be unique), and $\eta_t > 0$ is the step-size. 
Given the gradient $\nabla{}f(\X_t)$, the runtime to compute $\X_{t+1}$ is dominated by the computation of the projection. It is well known that the Euclidean projection is given by the projection matrix which corresponds to the span of the top $k$ eigenvectors of the matrix $\X_t-\eta_t\nabla{}f(\X_t)$. While accurate computation of this projection requires a (thin) singular value decomposition (SVD) of a $n\times n$ matrix, which amounts to $O(n^3)$ runtime, it can also be approximated up to sufficiently small error using fast iterative methods, such as the well-known orthogonal iteration method \cite{golub1996matrix} (aka subspace iteration method \cite{saad2011numerical}). The orthogonal iteration method finds a $n\times k$ matrix $\Q$ with orthonormal columns which approximately span the subspace spanned by the $k$ leading eigenvectors of a given positive semidefinite $n\times n$ matrix $\A$, by repeatedly applying the iterations: $(\Q,\R)\gets\textsc{QR-Factorize}(\A\Q)$, where $\textsc{QR-factorize}(\cdot)$ denotes the QR factorization of a matrix, i.e., $\Q\in\reals^{n\times k}$ has orthonormal columns. Every iteration of this method takes in worst case only $O(kn^2)$ time. When the gradient $\nabla{}f(\X_t)$ admits a favorable structure such as sparsity or a low-rank factorization, the runtime to approximate the projection onto $\mP_{n,k}$ using the orthogonal iteration method could be significantly improved.

Another natural approach to tackle Problem \eqref{eq:optProbNonconv} is to exploit the fact that each $\X\in\mP_{n,k}$ could be factored as $\X=\Q\Q^{\top}$, $\Q\in\reals^{n\times k}$ having orthonormal columns, and to apply gradient steps w.r.t. this factorization. This leads to the following dynamics, which we refer to as \textit{Gradient Orthogonal Iteration}:
\begin{align}\label{eq:QRgrad}
&\Z_{t+1} \gets \Q_t - \eta_t\frac{\partial{}f(\Q\Q^{\top})}{\partial\Q}\Big|_{\Q_t}  = \Q_t -\eta_t\nabla{}f(\Q_t\Q_t^{\top})\Q_t,  
  \nonumber \\ 
&(\Q_{t+1},\mathbf{R}_{t+1}) \gets \textsc{QR-factorize}(\Z_{t+1}),
\end{align} 
where the QR-factorization step is required to ensure that $\Q_{t+1}\Q_{t+1}^{\top}$ is also a projection matrix. 

As opposed to the Dynamics \eqref{eq:nonconvProjGrad}, which as discussed, an efficient implementation of will require to run a QR-based iterative method to compute the Euclidean projection onto $\mP_{n,k}$ on each iteration, the Dynamics \eqref{eq:QRgrad} only requires a single QR factorization per iteration, and thus, given the gradient matrix $\nabla{}f(\Q_t\Q_t^{\top})$, the next iterate $\Q_{t+1}$ can be computed in overall $O(n^2k)$ time. As mentioned above, this runtime could be further significantly improved if the multiplication $\nabla{}f(\Q_t\Q_t)\Q_t$ could be carried out faster than $O(n^2k)$ (for instance when the gradient is sparse or admits a low-rank factorization), since all other operations require only $O(k^2n)$ time (e.g., factorizing of $\Z_{t+1}$). 

\begin{center}
\textit{Obtaining provable guarantees on the fast local convergence of the Dynamics \eqref{eq:QRgrad} to a global optimal solution of Problem \eqref{eq:optProbNonconv} is the main contribution of this work.}
\end{center}
While both Dynamics \eqref{eq:nonconvProjGrad}, \eqref{eq:QRgrad} apply efficient iterations, since they are inherently nonconvex, arguing about their convergence to a global optimal solution of \eqref{eq:optProbNonconv} is difficult in general. An alternative is to replace Problem \eqref{eq:optProbNonconv} with a convex counterpart, for which, arguing about the convergence of first-order methods to a global optimal solution is well understood.  Consider the convex set $\mF_{n,k} = \conv(\mP_{n,k})$, where $\conv(\cdot)$ denotes the convex-hull operation. $\mF_{n,k}$ is also called the \textit{Fantope} and it is known to admit the following important  characterization: $\mF_{n,k} = \{\X\in\mbS^n~|~\I\succeq\X\succeq{}0, \trace(\X)=k\}$, where $\A\succeq 0$ denotes that $\A$ is a positive semidefinite matrix (PSD), see for instance \cite{overton1992sum}. This leads to the convex problem:
\begin{align}\label{eq:optProbConv}
\min_{}f(\X) \quad \textrm{subject to} \quad  \X\in\mF_{n,k}  = \{\X\in\mbS^n ~|~\I\succeq\X\succeq{}0,\trace(\X)=k\}.
\end{align}
A well known first-order method applicable to \eqref{eq:optProbConv} is the Frank-Wolfe method (aka conditional gradient) \cite{jaggi2013revisiting}, which for the convex  Problem \eqref{eq:optProbConv} follows the dynamics:
\begin{align}\label{eq:FrankWolfe}
\V_{t} &\gets \argmin_{\V\in\mP_{n,k}}\trace(\V\nabla{}f(\X_t)), \quad\X_{t+1} \gets (1-\eta_t)\X_t + \eta_t\V_{t}, ~\eta_t\in[0,1].
\end{align}
It follows from Ky Fan's maximum principle \cite{fan1949theorem} that computing $\V_{t}$ amounts to computing the projection matrix onto the span of the $k$ eigenvectors corresponding to the $k$ smallest eigenvalues of $\nabla{}f(\X_t)$, and hence can be carried out efficiently using the orthogonal iterations method or similar methods, similarly to the computation of the projection in \eqref{eq:nonconvProjGrad}) discussed above. \footnote{We note that one can also consider projection-based first-order methods for Problem \eqref{eq:optProbConv}, such as the projected gradient method, however in general, the projection onto the Fantope $\mF_{n,k}$ will not be a low-rank matrix and hence its computation will require an expensive SVD computation (see details in the sequel).} Note however that the Frank-Wolfe iterates will not be, in general, low rank, and only yield a $O(1/t)$ convergence rate \cite{jaggi2013revisiting}.

\subsection{The eigengap assumption and strict complementarity}
We now turn to discuss our only non-completely standard assumption on Problems \eqref{eq:optProbNonconv}, \eqref{eq:optProbConv}, which will underly all of our contributions, and in particular will facilitate our local linear convergence rates.

\begin{assumption}[Main assumption]\label{ass:gap}
An optimal solution $\X^*$ to the convex Problem \eqref{eq:optProbConv} is said to satisfy the eigen-gap assumption with parameter $\delta >0$, if $\lambda_{n-k}(\nabla{}f(\X^*)) - \lambda_{n-k+1}(\nabla{}f(\X^*)) \geq \delta$. 
\end{assumption}


Assumption \ref{ass:gap} in particular implies the following theorem which states that the convex relaxation \eqref{eq:optProbConv} exactly recovers the unique and optimal solution to the nonconvex Problem \eqref{eq:optProbNonconv}. This is one aspect in which  Assumption \ref{ass:gap} captures ``well-posed'' instances of Problem \eqref{eq:optProbNonconv}. The proof is  in the appendix.

\begin{theorem}\label{thm:optSol}
If an optimal solution $\X^*$ to Problem  \eqref{eq:optProbConv} satisfies Assumption \ref{ass:gap} with some parameter $\delta >0$, then it has rank $k$, i.e., $\X^*\in\mP_{n,k}$, and it is the unique optimal solution to both Problem \eqref{eq:optProbConv} and Problem \eqref{eq:optProbNonconv}.
\end{theorem}

Assumption \ref{ass:gap} is  tightly related to the convex Problem \eqref{eq:optProbConv} through the concept of \textit{strict-complementarity}, which is a classical concept in constrained continuous optimization theory  \cite{alizadeh1997complementarity}. A similar connection between an eigengap in the gradient at an optimal solution and strict complementarity has been already established in \cite{ding2020spectral} for low-rank matrix optimization problems, where the underlying convex set is either the nuclear norm ball of matrices or the set of PSD matrices with unit trace. Now we establish a similar relationship for the convex relaxation \eqref{eq:optProbConv} and the Fantope, which is slightly more involved.
Let us write the Lagrangian of the convex Problem \eqref{eq:optProbConv}:
\begin{align*}
L(\X,\Z_1,\Z_2,s) = f(\X)-\langle \Z_{1}, \X \rangle - \langle \Z_{2}, \I-\X  \rangle - s(\trace(\X)-k),
\end{align*}
where the dual matrix variables $\Z_1,\Z_2$ are constrained to be PSD, i.e., $\Z_1\succeq{}0,\Z_2\succeq{}0$.

The KKT conditions state that $\X^*$, $(\Z_1^*,\Z_2^*,s^*)$ are corresponding optimal primal-dual solutions if and only if the following conditions hold:
\begin{align*}
&1.~ \I\succeq\X^*\succeq 0, \trace(\X^*) =k, \Z_1^*\succeq 0, \Z_2^*\succeq 0, \quad 2.~ \nabla{}f(\X^*) = \Z_1^* - \Z_2^* + s^*\I, \\
&3.~ \langle{\Z_1^*,\X^*}\rangle = \langle{\Z_2^*,\I-\X^*}\rangle = 0.
\end{align*}
Condition 3 is known as \textit{complementarity}. Since $ \Z_{1}^{*}, \Z_{2}^{*}$ are PSD and $ 0 \preceq \X^{*} \preceq \I$, this further implies that
$\Z_{1}^{*}\X^{*}= \mathbf{0},  \Z_{2}^{*}(\I-\X^{*})=\mathbf{0}$, 
which in turn implies that
\begin{align*}
\textrm{range}(\X^{*})\subseteq \textrm{nullspace}(\Z_{1}^{*})~ \wedge ~\textrm{range}(\I-\X^{*})\subseteq \textrm{nullspace}(\Z_{2}^{*}).
\end{align*}

\begin{definition} \label{strict_comp_def}
A pair of primal-dual solutions $\X^*$, $(\Z_1^*,\Z_2^*,s^*)$ for Problem \eqref{eq:optProbConv} is said to satisfy \textit{strict complementarity}, if 
$\textrm{range}(\X^{*})=\textrm{nullspace}(\Z_{1}^{*}) \vee \textrm{range}(\I-\X^{*})= \textrm{nullspace}(\Z_{2}^{*})$,
which is the same as:
$\rank(\Z_{1}^{*}) = n-\rank(\X^*) \vee \rank(\Z_{2}^{*}) = \rank(\X^*)$.
\end{definition}
\begin{theorem}\label{thm:strictComp}
If an optimal solution $\X^*$ for Problem \eqref{eq:optProbConv} with $\rank(\X^*) = k$ satisfies strict complementarity for some corresponding dual solution, then $\lambda_{n-k}(\nabla{}f(\X^*)) -$
\newline $\lambda_{n-k+1}(\nabla{}f(\X^*)) > 0$.  Conversely, if an optimal solution $\X^*$ for Problem \eqref{eq:optProbConv} satisfies 
\newline $\lambda_{n-k}(\nabla{}f(\X^*)) - \lambda_{n-k+1}(\nabla{}f(\X^*)) > 0$, then it satisfies strict complementarity for every 
\newline corresponding dual solution.

\end{theorem}
The proof is given in the appendix.
Strict complementarity has played a central role in several recent works, both for establishing linear convergence rates for first-order methods, e.g., \cite{zhou2017unified, drusvyatskiy2018error, garber2019linear, ding2020spectral}, and improving the runtime of projected gradient methods due to SVD computations, for low-rank matrix optimization problems, e.g., \cite{garber2021convergence, kaplan2021low}.  

\subsection{Additional related work}
Efficient gradient methods for low-rank nonconvex optimization have received significant interest in recent years, here we mention only a few. \cite{pmlr-v49-bhojanapalli16, park2018finding} gave deterministic guarantees on  the local convergence to a global minimizer of factorized gradient descent for certain low-rank optimization problems, under the mild assumption that a low-rank global minimizer exists. However, these results cannot capture constraints such as those in our Problem \eqref{eq:optProbNonconv} which encode projection matrices. 
\cite{chen2015fast}, which considers nonconvex gradient methods for low-rank statistical estimation, also considers constraints  that cannot capture projection matrices as in Problem \eqref{eq:optProbNonconv}. An exception is a specific case they
consider of linear objective functions. Moreover, even for linear functions such as the specific sparse PCA objective they consider, their analysis requires several non-trivial conditions to hold (e.g. local descent, local
smoothness etc), which they only show to hold under Gaussian data. 
\subsection{Notation}
Throughout this work we let $\Vert{\cdot}\Vert$ denote the Euclidean norm for vectors in $\reals^n$ and the spectral norm (largest singular value) for matrices in $\reals^{m\times n}$ or $\mbS^n$. We let $\Vert{\cdot}\Vert_F$ denote the Frobenius (Euclidean) norm for matrices. For a matrix $\X\in\mbS^n$, we let $\lambda_i(\X)$ denote the $i$th largest eigenvalue of $\X$. We let $\langle{\cdot,\cdot}\rangle$ denote the standard inner-product for both spaces $\reals^n$ and $\mbS^n$.  

\section{Overview of Results}
\subsection{Main result}
Our main novel contribution is the proof of the following theorem regarding the local linear convergence of the gradient orthogonal iteration  \eqref{eq:QRgrad} to the  optimal solution of Problems \eqref{eq:optProbConv}, \eqref{eq:optProbNonconv}.
\begin{theorem} \label{qr_theorem}[Local linear convergence of gradient orthogonal iteration]
Suppose Assumption \ref{ass:gap} holds true for some optimal solution $\X^*$ to Problem \eqref{eq:optProbConv} with some parameter $\delta>0$. Let $G \geq \sup_{\X\in\mF_{n,k}}\Vert{\nabla{}f(\X)}\Vert$. Consider the sequence $\{\Q_t\}_{t\geq 1}$ generated by Dynamics \eqref{eq:QRgrad} with a fixed step-size $\eta_t = \eta = \frac{1}{5 \max\{\beta, G\}}$ for all $t\geq 1$, and when initialized with $\Q_1\in\mP_{n,k}$ such that
$\Vert{\Q_{1}\Q_{1}^{\top}-\X^*}\Vert_F \leq \min\{1,\sqrt{\frac{\delta}{2}}\} \frac{\eta\delta}{2(1+\eta\beta)}$.
Then, we have that
\begin{align*}
\forall t\geq 1:\quad f(\Q_{t}\Q_{t}^{\top}) - f(\X^*) \leq \big(f(\Q_{1}\Q_{1}^{\top})-f(\X^*)\big)\exp \Big(-\frac{\delta (t-1)}{40 \max\{\beta,G\}}\Big).
\end{align*}
\end{theorem}

While, as stated above, this is not the first work to consider strict complementarity conditions for bridging convex and nonconvex methods for low-rank optimization, previous works such as \cite{garber2019linear, ding2020spectral, garber2021convergence, kaplan2021low} consider gradient methods that rely on (nearly) exact (low-rank) SVD computations on each iteration, whereas Theorem \ref{qr_theorem} considers the more efficient \textit{SVD-free} Dynamics \eqref{eq:QRgrad}, that requires only a single QR-factorization of an $n\times k$ matrix per iteration, which is much faster and simpler to implement. Accordingly, the proof is also considerably more challenging and requires new ideas.

\subsection{Additional results}

We also prove the following two theorems regarding the local linear convergence of the projected gradient Dynamics \eqref{eq:nonconvProjGrad} and the Frank-Wolfe Dynamics \eqref{eq:FrankWolfe}. These extend the results in \cite{garber2021convergence, garber2019linear} from optimization over the set of positive semidefinite matrices with unit trace to the Fantope.
\begin{theorem}\label{pgd_theorem}[Local linear convergence of nonconvex PGD]
Suppose Assumption \ref{ass:gap} holds true for some optimal solution $\X^*$ to Problem \eqref{eq:optProbConv} with some parameter $\delta>0$.
Consider the sequence $\{\X_{t}\}_{t\geq 1}$ generated by Dynamics \eqref{eq:nonconvProjGrad} with a fixed step-size $\eta_t = \eta = 1/\beta$ for all $t\geq 1$, and when initialized with $\X_{1}\in\mF_{n,k}$ such that
$\|\X_{1}-\X^{*}\|_{F} \leq \frac{\delta}{4\beta}$. Then, for all $t\geq 1$ it holds that 
\begin{enumerate}
\item
$\rank(\X_{t+1}) = k$, and thus, given $\X_t$ and $\nabla{}f(\X_t)$, $\X_{t+1}$ can be computed via a rank-$k$ SVD,
\item
$f(\X_t)-f(\X^*) \leq \left({f(\X_1) - f(\X^*)}\right)\cdot\exp(-\Theta(\delta/\beta)(t-1)))$. 
\end{enumerate}
\end{theorem}


\begin{theorem}\label{fw_theorem}[Local linear convergence of Frank-Wolfe]
Suppose Assumption \ref{ass:gap} holds true for some optimal solution $\X^*$ to Problem \eqref{eq:optProbConv} with some parameter $\delta>0$.
Consider the sequences $\{(\X_{t},\V_t)\}_{t\geq 1}$ generated by Dynamics \eqref{eq:FrankWolfe} when $\eta_t$ is chosen via line-search. Then, there exists $T_0=O\left({k(\beta/\delta)^3}\right)$ such that, 
\vspace{-8pt}
\begin{align*}
\forall t\geq T_0:~ f(\X_{t+1}) - f(\X^*) \leq \big(f(\X_{t}) - f(\X^*)\big)\Big({1-\min\{ \dfrac{\delta}{12\beta}, \dfrac{1}{2} \}}\Big).
\end{align*}
Moreover, for all $t\geq 1$, the rank-$k$ matrix $\V_t$ satisfies 
$\|\V_t-\X^*\|_F^2=O\Big({\dfrac{\beta^2}{\delta^3}\big(f(\X_t) - f(\X^*)\big)}\Big)$.
\end{theorem}

\paragraph*{What if Assumption \ref{ass:gap} fails?}
In case Assumption \ref{ass:gap} does not hold or holds with negligible parameter $\delta$,  not all is lost, since by considering weaker versions of Assumption \ref{ass:gap}, which consider eigen-gaps between higher eigenvalues, we can still guarantee that $\X^*$ (an optimal solution to Problem \eqref{eq:optProbConv}) has low rank, and that at least the projected gradient method (when applied to Problem \eqref{eq:optProbConv}), locally, will require only a low-rank SVD to compute the projection onto the Fantope, while guaranteeing the standard convergence rate of $O(1/t)$ (not linear rate as when Assumption \ref{ass:gap} holds). 
\begin{theorem}\label{thm:pgd_general}
Let $\X^*\in\mF_{n,k}$ be some optimal solution to Problem \eqref{eq:optProbConv} and let $\mu_1 \geq \mu_2 \geq...\mu_n$ denote the eigenvalues of $-\nabla{}f(\X^*)$. Let $r$ be the smallest integer such that $r\geq k$ and $\mu_r - \mu_{r+1} > 0$. Then, it holds that $\rank(\X^*) \leq r$. Moreover, consider the projected gradient dynamics w.r.t. Problem \eqref{eq:optProbConv} given by,
$\X_{t+1} \gets \Pi_{\mF_{n,k}}[\X_t-\beta^{-1}\nabla{}f(\X_t)]$.
For any $r'\in\{r,\dots,n-1\}$, if $\Vert{\X_1 - \X^*}\Vert_F \leq \frac{\mu_k - \mu_{r'+1}}{4\beta}$, then it holds that,
\begin{enumerate}
\item
 $\forall{}t\geq 1$, $\rank(\X_{t+1}) \leq r'$, i.e., given $\X_t$ and $\nabla{}f(\X_t)$, $\X_{t+1}$ can be computed via a rank-$r'$ SVD.
\item
$\{\X_t\}_{t\geq 1}$ converges with the standard PGD rate:  $f(\X_t) - f(\X^*) = O(\beta\Vert{\X_1-\X^*}\Vert_F^2/t)$.
\end{enumerate}
\end{theorem}
\begin{remark}
Note that via the parameter $r'$, Theorem \ref{thm:pgd_general} offers a flexible tradeoff between the  radius of the ball in which PGD needs to be initialized in (increasing $r'$ increases the radius), and the rank of the iterates which in turn, implies an upper-bound on the rank of SVD computations required for the projection, which controls the runtime of each iteration. 
\end{remark}
\begin{remark}
Theorem \ref{thm:pgd_general} may be in particular interesting when $f(\cdot)$ is \textit{subspace-monotone} in the sense that for any two subspaces $\mS_1 \subseteq \mS_2\subseteq\reals^n$ and their corresponding projection matrices $\P_1,\P_2\in\mbS^n$, it holds that $f(\P_2)\leq f(\P_1)$. In this case, given an optimal solution $\X^*$ to the convex Problem \eqref{eq:optProbConv} with eigen-decomposition $\X^* = \sum_{i=1}^r\lambda_i\u_i\u_i^{\top}$, when $k < r << n$, using a projection matrix $\P^* = \sum_{i=1}^r\u_i\u_i^{\top}$  which satisfies $f(\P^*) \leq \min_{\X\in\mP_{n,k}}f(\X)$ may be of interest. For instance, it is not hard to show that $f(\cdot)$ of the form $f(\X) = \sum_{i=1}^mg_i(\Vert{\q_i-\X\q_i}\Vert)$, where $g_i(\cdot)$ is monotone non-decreasing and $\{\q_i\}_{i=1}^m\subset\reals^n$, is subspace-monotone.
\end{remark}

The complete proofs of Theorems \ref{qr_theorem}, \ref{pgd_theorem}, \ref{fw_theorem}, \ref{thm:pgd_general}, as well as additional results, are given in the appendix. Below we give the main ideas in the proof of Theorem  \ref{qr_theorem}.

\section{Proof Sketch of Theorem \ref{qr_theorem}}
\subsection{Preliminaries}\label{sec:analPrelim}

\begin{lemma}[Euclidean projection onto the Fantope]\label{lem:proj}
Let $\X\in \mbS^{n} $ and consider its eigen decomposition $\X=\sum_{i=1}^{n}\gamma_i\u_i\u_i^{\top}$.
The Euclidean projection $\Pi_{\mF_{n,k}}[\X]$ is given by:
$\Pi_{\mF_{n,k}}[\X]=\sum_{i=1}^{n}\gamma^{+}_i(\theta)\u_i\u_i^{\top}$, 
where $\gamma^{+}_i(\theta)=\min(\max(\gamma_i-\theta,0),1)$, and $\theta$ satisfies the equation $\sum_{i=1}^n\gamma_i^{+}(\theta)=k$.
\smallskip Moreover, $\forall{}r\in\{k,...,n-1\}$ it holds that $rank(\Pi_{\mF_{n,k}}(\X))\leq r$ if and only if $\sum_{i=1}^r \min(\gamma_i-\gamma_{r+1},1)\geq k$.
\end{lemma}
\begin{remark}\label{remark:rank}
Lemma \ref{lem:proj} implies that if $\rank(\X) \leq r$, then only the top $r$ components in the SVD of $\X$ are needed to compute $\Pi_{\mF_{n,k}}[\X]$, i.e., a rank-$r$ SVD of $\X$. Moreover, given the rank-$(r+1)$ SVD, we can check the condition $\sum_{i=1}^r \min(\gamma_i-\gamma_{r+1},1)\geq k$, to verify whether the projection has rank $\leq r$.
\end{remark}




The following lemma lower bounds, under Assumption \ref{ass:gap}, the radius of the ball around the unique optimal solution $\X^{*}$ inside-which, the PGD mapping w.r.t. the Fantope $\mF_{n,k}$ with a fixed step-size,  is guaranteed to produce rank-$k$ matrices, i.e., matrices in $\mP_{n,k}$, which means that it coincides precisely with the PGD mapping w.r.t. the nonconvex set $\mP_{n,k}$, given by the Dynamics  \eqref{eq:nonconvProjGrad}. 
\begin{lemma}\label{lem:rad}
Let $\X^{*}\in \mF_{n,k}$ be an optimal solution to Problem \eqref{eq:optProbConv} which satisfies Assumption \ref{ass:gap} with some parameter $\delta >0$, and let $\eta >0$. For any $\X\in \mF_{n,k}$ which satisfies 
$\|\X-\X^*\|_F\leq\dfrac{\eta\delta}{2(1+\eta\beta)}$,
it holds that $\rank(\Pi_{\mF_{n,k}}[\X-\eta\nabla f(\X)])=k$.
\end{lemma}

The following lemma establishes that under Assumption \ref{ass:gap}, Problem \eqref{eq:optProbConv}  has a quadratic growth property. This property is known to facilitate linear convergence rates of gradient methods \cite{necoara2019linear, karimi2016linear}.

\begin{lemma}[Quadratic Growth]\label{lem:quad}
Let $\X^*\in \mF_{n,k}$ be an optimal solution to Problem \eqref{eq:optProbConv} for which Assumption \ref{ass:gap} holds with some  $\delta >0$. Then, 
$\forall \X\in \mF_{n,k}:$ $\| \X-\X^{*}\| _{F}^{2} \leq \dfrac{2}{\delta} (f(\X)-f(\X^{*})) $.
\end{lemma}

\subsection{Gradient Orthogonal Iteration Analysis}\label{sec:GOI}
We outline the proof of our main algorithmic result --- the local linear convergence result of the gradient orthogonal iteration \eqref{eq:QRgrad} given in Theorem \ref{qr_theorem}. For convenience, we rewrite the Dynamics \eqref{eq:QRgrad} as Algorithm \ref{alg:QR} below which also introduces notation that will be helpful throughout the analysis. Throughout this section we also introduce the auxiliary sequence $\{\X_t\}_{t\geq 1}\subset\mF_{n,k}$ given by: $\X_1 = \Y_1$ and $\X_{t+1} = \Pi_{\mF_{n,k}}[\Y_t - \eta\nabla{}f(\Y_t)]$ for all $t\geq 1$.

At a  high-level, our analysis of Algorithm \ref{alg:QR} relies on the following two components:
\begin{enumerate}
\item
Using Lemma \ref{lem:rad} we can argue that, in the proximity of $\X^*$, $\rank(\X_t) = k$, i.e., $\X_t\in\mP_{n,k}$. This  implies that $\X_t$ is the projection matrix onto the span of top $k$ eigenvectors of $\W_{t}$.
\item
We view $\Q_t$ as the outcome of applying one iteration of the  orthogonal iterations method  \cite{golub1996matrix, saad2011numerical} to $\W_t$ (see also discussion in the Introduction). Combined with the previous point, this allows  to argue that $\Y_t = \Q_t\Q_t^{\top}$ is sufficiently close to the  projected gradient update $\X_t$, which drives the convergence. 
\end{enumerate}
\begin{algorithm}
\caption{Gradient Orthogonal Iteration}\label{alg:QR}
\begin{algorithmic}[1]
\State initialization: $\Y_1 = \Q_1\Q_1^{\top}$ for some $\Q_1\in\reals^{n\times k}$ such that $\Q_1^{\top}\Q_1 = \I$
\For{$t=1,2... $}
\State $\W_{t+1}\gets \Y_{t}-\eta\nabla f(\Y_{t})$
\State $(\Q_{t+1}, \R_{t+1}) \gets \textsc{QR-factorize}(\W_{t+1}\Q_{t})$ (that is $\Q_{t+1}\R_{t+1}=\W_{t+1}\Q_{t}$)
\State $\Y_{t+1}\gets \Q_{t+1}\Q_{t+1}^{\top}$
\EndFor
\end{algorithmic}
\end{algorithm}

The following key lemma establishes the connection between the sequence $\{\Y_t\}_{t\geq 1}$ produced by Algorithm \ref{alg:QR}, and the corresponding sequence of exact projected gradient steps $\{\X_t\}_{t\geq 1}$. The proof relies on an original extension of the classical orthogonal iteration method (see  \cite{golub1996matrix}).
\begin{lemma}\label{lem:QRit}
Fix some iteration $t\geq 1$. Suppose that $\eta < 1/G$, $\X_{t+1}\in\mP_{n,k}$, and $\Vert{\X_{t+1}-\Y_t}\Vert_F < \sqrt{2}$. It holds that,
$\Vert{\X_{t+1}-\Y_{t+1}}\Vert_F^2 \leq \frac{1}{1-\frac{1}{2}\Vert{\X_{t+1}-\Y_t}\Vert_F^2}\Big({\frac{\eta G}{1-\eta G}}\Big)^{2}\Vert{\X_{t+1}-\Y_t}\Vert_F^2$.
\end{lemma}
\begin{proof}[Proof of Lemma \ref{lem:QRit}]
Let us write the eigen-decomposition of $\W_{t+1}=\Y_t - \eta\nabla{}f(\Y_t)$ as:
\begin{align*}
\W_{t+1} = \V\Lambda\V^{\top} =
\begin{bmatrix}
    \V_1  &  \V_2     
\end{bmatrix}
\begin{bmatrix}
    \Lambda_{1}  &  0      \\
    0  &  \Lambda_{2}      
\end{bmatrix}
\begin{bmatrix}
    \V_{1}^{\top}       \\
    \V_{2}^{\top}       
\end{bmatrix},
\end{align*}
where $\V_1\in\reals^{n\times k},\Lambda_1\in\reals^{k\times k}$ correspond to the largest $k$ eigenvalues.

The main part of the proof will be to prove that
$ \|\V_{2}^{\top}\Q_{t+1}\|_{F}^{2}\leq  \frac{1}{\sigma^2_{\min}(\V_1^{\top}\Q_t)}\left({\frac{\eta G}{1-\eta G}}\right)^{2}\|\V_{2}^{\top}\Q_{t}\|_{F}^{2}$.

Note that by definition of $\X_{t+1}$ we have that,    
\begin{align*}
\X_{t+1}&=\argmin_{\X\in\mF_{n,k}}\Vert{\X - \W_{t+1}}\Vert_F^2 \underset{(a)}{=} \argmin_{\X\in\mP_{n,k}}\Vert{\X - \W_{t+1}}\Vert_F^2\underset{(b)}{=} \argmax_{\X\in\mP_{n,k}}\langle{\X,\W_{t+1}}\rangle = \V_1\V_1^{\top},
\end{align*}
where (a) follows from the assumption of the lemma that $\X_{t+1}\in\mP_{n,k}$, and (b) follows since all matrices in $\mP_{n,k}$ have the same Frobenius norm.

This further implies that
{\small
\begin{align}\label{eq:QRit:0}
\sigma^2_{\min}(\V_1^{\top}\Q_t) &= \lambda_k(\V_1^{\top}\Q_t\Q_t^{\top}\V_1) = \sum_{i=1}^k\lambda_i(\V_1^{\top}\Q_t\Q_t^{\top}\V_1) - \sum_{j=1}^{k-1}\lambda_j(\V_1^{\top}\Q_t\Q_t^{\top}\V_1)  \nonumber \\
&\geq \trace(\V_1^{\top}\Q_t\Q_t^{\top}\V_1) - (k-1)\lambda_1(\V_1^{\top}\Q_t\Q_t^{\top}\V_1) \geq\trace(\X_{t+1}\Y_t) - (k+1)  \nonumber \\
&= \left({k - \frac{1}{2}\Vert{\X_{t+1}-\Y_t}\Vert_F^2}\right) - (k-1) = 1 -  \frac{1}{2}\Vert{\X_{t+1}-\Y_t}\Vert_F^2.
\end{align}}
Thus, under the assumption that $\Vert{\X_{t+1}-\Y_t}\Vert_F <\sqrt{2}$, we have that $(\V_1^{\top}\Q_t)$ is invertible.

Since $(\Q_{t+1},\R_{t+1})$ is the QR factorization of $\W_{t+1}\Q_t$, using the eigen-decomposition of $\W_{t+1}$ we can write
$\Q_{t+1}\R_{t+1}=\V\Lambda\V^{\top}\Q_{t}$.
Multiplying both sides from the left by $\V^{\top}$   we get,
\begin{align*}
\begin{bmatrix}
    \V_{1}^{\top}\Q_{t+1}  \\
    \V_{2}^{\top}\Q_{t+1}     
\end{bmatrix}
\R_{t+1}=
\begin{bmatrix}
    \Lambda_{1}  &  0      \\
    0  &  \Lambda_{2}      
\end{bmatrix}
\begin{bmatrix}
    \V_{1}^{\top}\Q_{t}  \\
    \V_{2}^{\top}\Q_{t}     
\end{bmatrix}, 
\end{align*}
which leads to the two equations:
\begin{align}
\Lambda_{1}\V_{1}^{\top}\Q_{t}=\V_{1}^{\top}\Q_{t+1}\R_{t+1}, \label{eq:QRit:1st}\\
\Lambda_{2}\V_{2}^{\top}\Q_{t}=\V_{2}^{\top}\Q_{t+1}\R_{t+1}. \label{eq:QRit:2nd}
\end{align}

Under the assumption that $\eta < 1/G$, using Weyl's inequality we have that $\lambda_k(\W_{t+1}) 
\newline \geq \lambda_k(\Y_t) - \eta\lambda_1(\nabla{}f(\Y_t)) > 0$, and so $\Lambda_1$ is invertible. Since from \eqref{eq:QRit:0} we have that $\sigma_{\min}(\V_1^{\top}\Q_{t}) >0$, it follows that $\rank(\Lambda_1\V_1^{\top}\Q_t) = k$ and thus, from Equation \eqref{eq:QRit:1st} we have that $\V_1^{\top}\Q_{t+1}$ and $\R_{t+1}$ are both invertible and we can write
$\R_{t+1}=(\V_1^{\top}\Q_{t+1})^{-1}\Lambda_{1}\V_{1}^{\top}\Q_{t}$.

Multiplying both sides of Equation \eqref{eq:QRit:2nd} from the right with $\R_{t+1}^{-1}$, we get 
\begin{align*}
\V_{2}^{\top}\Q_{t+1}= \Lambda_{2}\V_{2}^{\top}\Q_{t}\left({(\V_1^{\top}\Q_{t+1})^{-1}\Lambda_{1}\V_{1}^{\top}\Q_{t}}\right)^{-1} = \Lambda_{2}\V_{2}^{\top}\Q_{t}(\V_{1}^{\top}\Q_{t})^{-1}\Lambda_{1}^{-1}\V_{1}^{\top}\Q_{t+1}.
\end{align*}

Now we can use this to bound $\|\V_{2}^{\top}\Q_{t+1}\|_{F}^{2}$:
\begin{align}\label{eq:QRit:1}
    &\|\V_{2}^{\top}\Q_{t+1}\|_{F}^{2} = \|\Lambda_{2}\V_{2}^{\top}\Q_{t}(\V_{1}^{\top}\Q_{t})^{-1}\Lambda_{1}^{-1}\V_{1}^{\top}\Q_{t+1}\|_{F}^{2} \nonumber \\
    &\underset{(a)}{\leq} \Vert{(\V_1^{\top}\Q_t)^{-1}}\Vert_2^2\Vert{\Lambda_1^{-1}}\Vert_2^2\Vert{\V_1^{\top}\Q_{t+1}}\Vert_2^2\Vert{\Lambda_2}\Vert_2^2\Vert{\V_2^{\top}\Q_t}\Vert_F^2 \underset{(b)}{\leq} \dfrac{\|\V_{2}^{\top}\Q_{t}\|_{F}^{2}}{\sigma_{\min}^{2}(\V_{1}^{\top}\Q_{t})}\left({\dfrac{\lambda_{k+1}(\W_{t+1})}{\lambda_{k}(\W_{t+1})}}\right)^{2},
\end{align}
where (a) follows from the inequalities $\Vert{\A\B}\Vert_F \leq \min\{\Vert{\A}\Vert_F\Vert{\B}\Vert_2, \Vert{\A}\Vert_2\Vert{\B}\Vert_F\}$,
\newline $\Vert{\A\B}\Vert_2 \leq \Vert{\A}\Vert_2\Vert{\B}\Vert_2$, and (b) follows from the eigen-decomposition of $\W_{t+1}$ and by noting that since $\V_1,\Q_{t+1}$ both have orthonormal columns, it holds that $\Vert{\V_1^{\top}\Q_{t+1}}\Vert_2 \leq 1$.

We upper-bound $\lambda_{k+1}(\W_{t+1})/\lambda_{k}(\W_{t+1})$ by using Weyl's inequality as follows:
\begin{align}\label{eq:QRit:4}
\dfrac{\lambda_{k+1}(\W_{t+1})}{\lambda_{k}(\W_{t+1})}\leq \dfrac{\lambda_{k+1}(\Y_{t})+\eta \lambda_{1}(-\nabla f(\Y_{t}))}{\lambda_{k}(\Y_{t})+\eta \lambda_{n}(-\nabla f(\Y_{t}))}\leq \dfrac{\eta G}{1-\eta G},
\end{align}
where we have used the fact that $\Y_t\in\mP_{n,k}$, and so $\lambda_k(\Y_t)=1,\lambda_{k+1}(\Y_t) = 0$.

Plugging \eqref{eq:QRit:4} into \eqref{eq:QRit:1} we indeed obtain,
\begin{align}\label{eq:QRit:main}
 \|\V_{2}^{\top}\Q_{t+1}\|_{F}^{2}\leq  \frac{1}{\sigma^2_{\min}(\V_1^{\top}\Q_t)}\left({\frac{\eta G}{1-\eta G}}\right)^{2}\|\V_{2}^{\top}\Q_{t}\|_{F}^{2}.
\end{align}

Now, for the final part of the proof, we note that $\Vert{\V_2^{\top}\Q_{t+1}}\Vert_F^2 = \trace(\V_2\V_2^{\top}\Y_{t+1}) =
\newline\trace((\I-\X_{t+1})\Y_{t+1}) = k - \trace(\X_{t+1}\Y_{t+1}) = \frac{1}{2}\Vert{\X_{t+1}-\Y_{t+1}}\Vert_F^2$, and similarly, $\Vert{\V_2^{\top}\Q_t}\Vert_F^2 = \frac{1}{2}\Vert{\X_{t+1}-\Y_t}\Vert_F^2$. Plugging these observations and \eqref{eq:QRit:0} into \eqref{eq:QRit:main}, we obtain the lemma. 
\end{proof}

The following lemma is the main step in the proof of the convergence rate of Algorithm \ref{alg:QR}.
\begin{lemma}\label{lem:errorDec}
Let us denote $h_{t}=f(\Y_{t})-f(\X^*)$ for all $t\geq 1$. Fix some iteration $t$ of Algorithm \ref{alg:QR}, and
suppose that $\eta \leq \frac{1}{5\max\{\beta,G\}}$, $\X_{t+1}\in\mP_{n,k}$, and that  $\Vert{\X_{t+1}-\Y_{t}}\Vert_F \leq 1$.  Denote the constants 
$C_0 =  2\Big({\frac{\eta G}{1-\eta G}}\Big)^{2}$, $C_1 =  \frac{2(1+\eta{}G)C_0}{1-2\eta\beta-2C_0(1+\eta{}G)}$.
 It holds that, 
$h_{t+1} \leq  \Big({1 - \frac{\eta\delta}{4(1+C_1)}}\Big)h_t$, where $\delta>0$ is the constant from Assumption \ref{ass:gap}.
\end{lemma}
\begin{proof}
Using the $\beta$-smoothness of  $f(\X)$, for any $\X\in\mF_{n,k}$ and $\eta\leq \dfrac{1}{\beta}$ it holds that
\begin{align}\label{eq:lem:err:1}
f(\X) &  \leq f(\Y_t) + \langle{\X-\Y_t,\nabla{}f(\Y_t)}\rangle + \frac{1}{2\eta}\Vert{\X-\Y_t}\Vert_F^2 \nonumber\\
&\underset{(a)}{\leq} f(\Y_t) + \langle{\X-\Y_t,\nabla{}f(\Y_t)}\rangle + \eta^{-1}\langle{\Y_t,\Y_t-\X}\rangle \nonumber \\
&= f(\Y_t) + \eta^{-1}\langle{\Y_t-\X,\Y_t -\eta\nabla{}f(\Y_t)}\rangle,
\end{align}
where (a) follows since using the fact that $\Y_t\in\mP_{n,k}$, we have that for any $\X\in\mF_{n,k}$ it holds that $\Vert{\X}\Vert_F^2 \leq k = \Vert{\Y_t}\Vert_F^2 = \langle{\Y_t,\Y_t}\rangle$.


Since $\X_{t+1} = \Pi_{\mF_{n,k}}[\Y_t-\eta\nabla{}f(\Y_t)] = \arg\min_{\X\in\mF_{n,k}}\Vert{\X-(\Y_t-\eta\nabla{}f(\Y_t))}\Vert_F^2$, and by the assumption of the lemma that $\X_{t+1}\in\mP_{n,k}$, using the first-order optimality condition, it can be shown that for all $\Z\in\mF_{n,k}$: $\langle{\X_{t+1}-\Z,\Y_t-\eta_t\nabla{}f(\Y_t)}\rangle \geq 0$, see Lemma \ref{lem:opt}. This implies that for all $\Z\in\mF_{n,k}$:
\begin{align}\label{eq:lem:err:2}
&\langle \Y_{t+1},\Y_{t}-\eta\nabla f(\Y_{t}) \rangle = \langle \X_{t+1},\Y_{t}-\eta\nabla f(\Y_{t}) \rangle - \langle \X_{t+1}-\Y_{t+1},\Y_{t}-\eta\nabla f(\Y_{t}) \rangle \geq \nonumber\\
& \langle \Z,\Y_{t}-\eta\nabla f(\Y_{t}) \rangle - \langle \X_{t+1}-\Y_{t+1},\Y_{t}-\eta\nabla f(\Y_{t}) \rangle \geq \nonumber\\
& \langle \Z,\Y_{t}-\eta\nabla f(\Y_{t}) \rangle - \|\X_{t+1}-\Y_{t+1}\|_F^{2}\|\W_{t+1}\|_{2},
\end{align}
where the last inequality is due to Lemma \ref{lem:aux}, which uses again the facts that $\X_{t+1}\in\mP_{n,k}$ and   $\X_{t+1} = \arg\min_{\X\in\mF_{n,k}}\Vert{\X-(\Y_t-\eta\nabla{}f(\Y_t))}\Vert_F^2$, which in turn imply that $\X_{t+1}=\arg\max_{\X\in\mP_{n,k}}\langle{\X,\W_{t+1}}\rangle$, and recalling that $\W_{t+1}=\Y_t-\eta\nabla{}f(\Y_t)$.

Setting $\X = \Y_{t+1}$ in \eqref{eq:lem:err:1} and plugging-in \eqref{eq:lem:err:2}, we have that for any $\Z\in\mF_{n,k}$ it holds that,
\begin{align}\label{eq:lem:err:3}
f(\Y_{t+1}) &\leq f(\Y_t)  + \eta^{-1}\left({\langle{\Y_t-\Z,\Y_t -\eta\nabla{}f(\Y_t)}\rangle +\|\X_{t+1}-\Y_{t+1}\|_F^{2}\|\W_{t+1}\|_{2}}\right)\nonumber \\
&=f(\Y_t) + \langle{\Z-\Y_t,\nabla{}f(\Y_t)}\rangle + \frac{1}{2\eta}\Vert{\Z-\Y_t}\Vert_F^2 + \frac{1}{\eta}\|\X_{t+1}-\Y_{t+1}\|_F^{2}\|\W_{t+1}\|_{2} \nonumber \\
&\leq f(\Y_t) + \langle{\Z-\Y_t,\nabla{}f(\Y_t)}\rangle + \frac{1}{2\eta}\Vert{\Z-\Y_t}\Vert_F^2 + \frac{1+\eta{}G}{\eta}\|\X_{t+1}-\Y_{t+1}\|_F^{2},
\end{align}
where the last inequality is due to the following upper-bound on $\Vert{\W_{t+1}}\Vert_2$:
\begin{align*}
\Vert{\W_{t+1}}\Vert_2 = \Vert{\Y_t - \eta\nabla{}f(\Y_t)}\Vert_2 \leq \Vert{\Y_t}\Vert_2 + \eta\Vert{\nabla{}f(\Y_t)}\Vert_2 \leq 1+ \eta{}G.
\end{align*}

In particular, setting $\Z=(1-\alpha)\Y_{t}+\alpha\X^{*}$ for some $\alpha\in[0,1]$, we get that
\begin{align*}
f(\Y_{t+1}) &\leq f(\Y_t) + \alpha\langle{\X^*-\Y_t,\nabla{}f(\Y_t)}\rangle + \frac{\alpha^2}{2\eta}\Vert{\X^*-\Y_t}\Vert_F^2 + \frac{1+\eta{}G}{\eta}\|\X_{t+1}-\Y_{t+1}\|_F^{2}.
\end{align*}
Subtracting $f(\X^*)$ from both sides, using the convexity of $f(\cdot)$, and Lemma \ref{lem:quad} gives
\begin{align*}
h_{t+1} &\leq \Big({1-\alpha+\dfrac{\alpha^2}{\eta\delta}}\Big)h_{t} + \frac{1+\eta{}G}{\eta}\|\X_{t+1}-\Y_{t+1}\|_F^{2}.
\end{align*}
Setting $\alpha=\eta\delta/2$ (note that since $\eta \leq 1/G$, we have that $\alpha\in[0,1]$), gives
\begin{align}\label{eq:lem:err:4}
h_{t+1} &\leq \Big({1-\frac{\eta\delta}{4}}\Big)h_{t} + \frac{1+\eta{}G}{\eta}\|\X_{t+1}-\Y_{t+1}\|_F^{2}.
\end{align}
We now continue to upper-bound the term $\|\X_{t+1}-\Y_{t+1}\|_F^{2}$. 
Using Lemma  \ref{lem:distXY}, which apply standard arguments in the analysis of first-order methods, that rely only on the facts that $\X_{t+1} = \Pi_{\mF_{n,k}}[\Y_t - \eta\nabla{}f(\Y_t)]$ and that $f(\cdot)$ is smooth and convex, we have that
\begin{align}\label{eq:lem:err:5}
\Vert{\X_{t+1}-\Y_t}\Vert_F^2 \leq \frac{\eta}{1-\eta\beta}\left({f(\Y_t) - f(\X_{t+1})}\right).
\end{align}

Let us set $\Z = \X_{t+1}$ in \eqref{eq:lem:err:3} to obtain that
\begin{align*}
f(\Y_{t+1}) &\leq f(\Y_t) + \langle{\X_{t+1}-\Y_t,\nabla{}f(\Y_t)}\rangle + \frac{1}{2\eta}\Vert{\X_{t+1}-\Y_t}\Vert_F^2 + \frac{1+\eta{}G}{\eta}\|\X_{t+1}-\Y_{t+1}\|_F^{2} \\
&\leq f(\X_{t+1}) + \frac{1}{2\eta}\Vert{\X_{t+1}-\Y_t}\Vert_F^2 + \frac{1+\eta{}G}{\eta}\|\X_{t+1}-\Y_{t+1}\|_F^{2},
\end{align*}
where the last inequality is due to convexity of $f(\cdot)$. Rearranging and using Lemma \ref{lem:QRit} along with the notation $C_0 =  2\left({\frac{\eta G}{1-\eta G}}\right)^{2}$, we have 
$f(\X_{t+1}) \geq f(\Y_{t+1}) -  \frac{1}{\eta}\Big({\frac{1}{2} + C_0(1+\eta{}G)}\Big)\Vert{\X_{t+1}-\Y_t}\Vert_F^2$.
Plugging into \eqref{eq:lem:err:5} we obtain
\begin{align*}
\Vert{\X_{t+1}-\Y_t}\Vert_F^2 \leq \frac{\eta}{1-\eta\beta}\Big({f(\Y_t) - f(\Y_{t+1}) +  \frac{1}{\eta}\Big({\frac{1}{2} + C_0(1+\eta{}G)}\Big)\Vert{\X_{t+1}-\Y_t}\Vert_F^2}\Big),
\end{align*}
and rearranging we obtain
\begin{align*}
\Vert{\X_{t+1}-\Y_t}\Vert_F^2 &\leq \frac{1}{1-\frac{1+2C_0(1+\eta{}G)}{2(1-\eta\beta)}}\frac{\eta}{1-\eta\beta}\left({f(\Y_t)-f(\Y_{t+1})}\right) = \frac{2\eta(h_t-h_{t+1})}{2(1-\eta\beta)-1-2C_0(1+\eta{}G)}.
\end{align*}
Using Lemma \ref{lem:QRit} again we have,
$\Vert{\X_{t+1}-\Y_{t+1}}\Vert_F^2 \leq \frac{2\eta{}C_0}{1-2\eta\beta-2C_0(1+\eta{}G)}\left({h_t-h_{t+1}}\right)$.
Plugging back into \eqref{eq:lem:err:4} we obtain
$h_{t+1} \leq \Big({1-\frac{\eta\delta}{4}}\Big)h_{t} + \frac{2(1+\eta{}G)C_0}{1-2\eta\beta-2C_0(1+\eta{}G)}\left({h_t - h_{t+1}}\right)$.
Denoting $C_1 =  \frac{2(1+\eta{}G)C_0}{1-2\eta\beta-2C_0(1+\eta{}G)}$, we finally obtain
$h_{t+1} \leq \frac{1}{1+C_1}\Big({1-\frac{\eta\delta}{4}+C_1}\Big)h_t = \Big({1 - \frac{\eta\delta}{4(1+C_1)}}\Big)h_t$,
as required.
The only thing  left is to choose a feasible step size. We have to require:
$1-2\eta\beta-2C_0(1+\eta{}G)>0$.
The latter holds for any $\eta \leq \frac{1}{5\max\{\beta,G\}}$.
\end{proof}
\section{Numerical Simulations}\label{sec:experiments}
We turn to discuss our numerical simulations. Some of the implementation details and results are deferred to the appendix.
We consider two models for robust recovery of a low-dimensional subspace from noisy samples: 1. a \textit{spiked covariance} model, and 2. a \textit{sparsely corrupted entries} model. In both models we minimize a robust loss based on the Huber function, which is convex and smooth, over the Fantope. We generate random instances and solve them to high precision (duality gap $<10^{-10}$) and take the resulting point $\X^*$ as the optimal solution. We measure the eigen-gap in $\nabla{}f(\X^*)$ (as in Assumption \ref{ass:gap}), and we compare the recovery error w.r.t. the naive PCA solution $\X_{PCA}$ which simply computes the leading subspace of the empirical covariance. 
The results are given in Table \ref{table:1}. As can be seen, for both models the recovery error is significantly lower than that of naive PCA, which demonstrates the usefulness of the chosen models . We see that  the eigen-gap assumption  indeed holds with substantial values of $\delta$.\\
\vspace{-6pt}
\begin{table*}[h!]\renewcommand{\arraystretch}{1.3}
{\footnotesize
\begin{center}
\begin{tabular}{ | P{7em} | P{5em} | P{5em}| P{4em} | P{4em} | P{5em} | P{4em}  |} 
  \hline
   Noise  prob. ($p$) & 0.05 & 0.1 & 0.2 & 0.3 & 0.4 & 0.5\\ 
\hline
\multicolumn{7}{|c|}{$\downarrow$ Model 1: spiked covariance $\downarrow$ }\\ \hline
  Eigen-gap ($\delta$) & 3.21& 2.87 & 2.36 & 2.04& 1.501 & 1.03\\
  \hline
  $\|\X^{*}-\P\|_{F}$ & $0.0047$ & $0.0075$ & $0.012$ & $0.016$ & $0.022$ &  $0.0298$\\ 
  \hline
  $\|\X_{PCA}-\P\|_{F}$ & 0.045 & 0.072 & 0.115&  0.157 & 0.212 & 0.292 \\ 
  \hline
 \multicolumn{7}{|c|}{$\downarrow$ Model 2: sparsely corrupted entries $\downarrow$ }\\ \hline
  Eigen-gap ($\delta$)& 5.72& 5.49 & 5.15 & 4.81& 4.38 &3.79\\
  \hline
  $\|\X^{*}-\P\|_{F}$ & $0.049$ & $0.067$ & $0.097$ & $0.111$ & $0.134$ &  $0.148$\\ 
  \hline
  $\|\X_{PCA}-\P\|_{F}$ & 0.148 & 0.199& 0.291&  0.335 & 0.401 & 0.439 \\ 
  \hline
\end{tabular}
\caption{Recovery and eigen-gap results for the spiked covariance and sparsely corrupted entries models with varying noise probabilities. $\P$ is the projection matrix onto the ground truth subspace. $n=100$, $k=10$, sample size $m=500$. Results are averages of 20 i.i.d. experiments.}\label{table:1}
\end{center}
}
\vskip -0.2in
\end{table*}\renewcommand{\arraystretch}{1} 

We additionally test the empirical convergence of nonconvex PGD (Dynamics  \eqref{eq:nonconvProjGrad}) and the gradient orthogonal iteration method (GOI, Dynamics \eqref{eq:QRgrad}) on the two models. We initialize both methods with the PCA solution $\X_{PCA}$ and use the same fixed step-size for both. We examine the convergence of both methods in terms of recovery error and approximation error (w.r.t. the objective function). Additionally, to showcase the benefit of avoiding exact SVD computations (as employed by nonconvex PGD) and using only a single QR factorization per iteration (as in GOI), we compare the runtimes of GOI and nonconvex PGD, but we exclude the time it takes to compute the gradient on each iteration and only account for the time it takes to perform either a $\rank-k$ SVD or a QR factorization, where both algorithms were implemented in Python and we have used the built-in functions \textsc{numpy.linalg.eigh} and \textsc{numpy.linalg.qr} to compute thin-SVDs and QR factorizations, respectively. Finally, we verify during the run of nonconvex PGD, that on each iteration, the projection onto $\mP_{n,k}$ is indeed the same as the projection onto the Fantope $\mF_{n,k}$ (see Remark \ref{remark:rank}), which suggests that nonconvex PGD indeed converges to the global minimum.

The results for the spiked covariance model are given in Figure \ref{fig:model1} (the results for the sparsely corrupted entries model are very similar and given in the appendix). It can be seen that indeed the distance between the iterates of the two methods decays very quickly and so the graphs of the recovery and approximation errors of both methods coincide. We  see that both methods demonstrate a linear convergence rate (w.r.t. the objective). We also see the significant savings in runtime when replacing a thin-SVD computation (in nonconvex PGD) with only a single QR factorization (in GOI).

\begin{figure}[h!]
	\centering
	\begin{minipage}[b]{0.24\textwidth}
		\includegraphics[width=\textwidth]{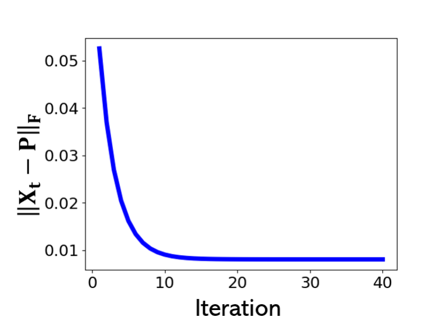}
	\end{minipage}
	\begin{minipage}[b]{0.24\textwidth}
		\includegraphics[width=\textwidth]{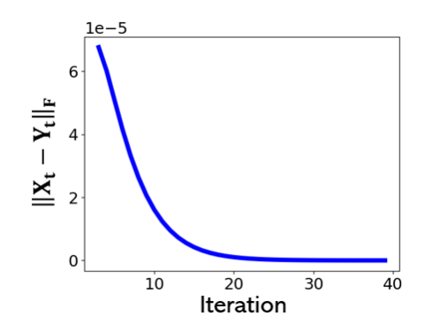}
	\end{minipage}
	\begin{minipage}[b]{0.24\textwidth}
		\includegraphics[width=\textwidth]{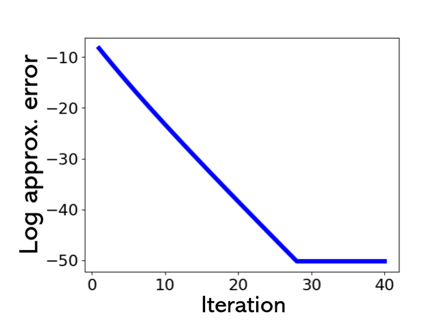}
	\end{minipage}
	\begin{minipage}[b]{0.24\textwidth}
		\includegraphics[width=\textwidth]{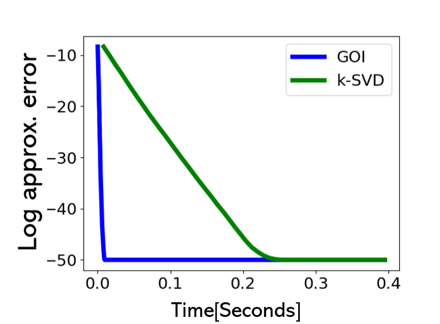}
	\end{minipage}
%
	 \caption{Convergence of PGD and GOI for the spiked covariance model with $p=0.1$. 1st and 3rd panels from the left show the recovery error ($\P$ is the ground truth projection matrix) and approximation error w.r.t. objective value of PGD, respectively. Convergence of GOI is omitted since it coincides with that of PGD. 2nd panel from the left shows the distance (in Frobenius norm) between the iterates of PGD $(\X_t)$ and those of GOI $(\Y_t)$. The rightmost panel shows the approximation error (in log scale) vs. time, when only the time to compute matrix factorizations is taken into account.} \label{fig:model1}\end{figure}
	 
\paragraph*{Importance of warm-start initalization:} We examine the performance of nonconvex PGD over $\mP_{n,k}$ (Dynamics \eqref{eq:nonconvProjGrad}) for the spiked covariance model considered above, but this time, when initialized with a random (uniformly distributed) projection matrix. We compare it with convex PGD which optimizes over the Fantope $\mF_{n,k}$ and uses a full-rank SVD to compute the projection. We use the same step-size as before. We see in Figure \ref{fig:random_2} (right panel) two trends. First, we clearly see that PGD w.r.t. $\mP_{n,k}$ and $\mF_{n,k}$ produce very different iterates which in particular implies that, as opposed to the case of warm-start initializaion, the projections of convex PGD onto the Fantope, throughout most of the run are not rank-$k$. Second, we see that nonconvex PGD is significantly slower than convex PGD. Thus, while both methods eventually converge to the same error, this suggests that far from a global minimizer, the behaviour of nonconvex gradient methods is indeed significantly different than in the local proximity of one, which supports the fact that our theoretical guarantees only hold in a local neighbourhood of a minimizer. 
\begin{figure}[h!]
	\centering
	\begin{minipage}[b]{0.280\textwidth}
		\includegraphics[width=\textwidth]{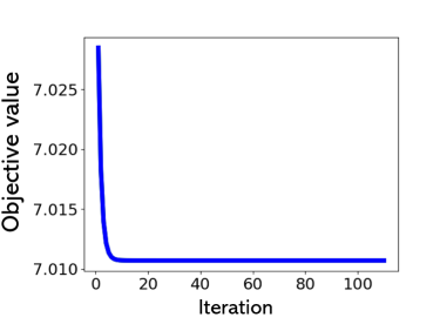}
	\end{minipage}
	\begin{minipage}[b]{0.28\textwidth}
		\includegraphics[width=\textwidth]{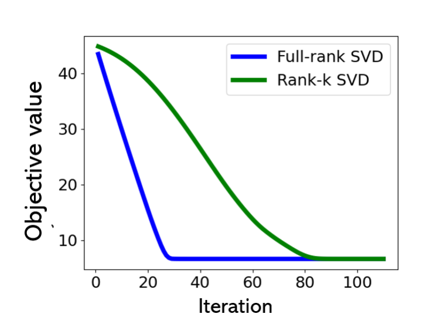}
	\end{minipage}
  \caption{Convergence of PGD for the spiked covariance model with $p=0.1$ over the Fantope $\mF_{n,k}$ with a full-rank SVD, and over $\mP_{n,k}$ with rank-$k$ SVD, when initialized with the PCA solution (left panel) and with random initialization (right panel). In the left panel the plots exactly coincide.} \label{fig:random_2}
\end{figure}
	
\bibliographystyle{plain}
\bibliography{bibs.bib}

\appendix

\section{Additional Details on Experiments}
The first robust recovery model we consider is a spiked covariance model, in which we draw a uniformly distributed projection matrix onto a $k$-dimensional subspace $\P \in \mP_{n,k}$, and we generate $m$ samples $\q_1,\dots,\q_m\in\reals^n$ such that for each $i\in[m]$, we set $\q_i =\P\z_i / \Vert{\P\z_i}\Vert$ with probability $1-p$, and $\q_ i = \z_i$ with probability $p$, where $p\in(0,0.5]$, and $\z_1,\dots,\z_m$ are i.i.d. uniformly distributed unit vectors. The goal is to recover $\P$ by minimizing the following objective function over $\mF_{n,k}$:
\begin{align*}
f(\X)=\sum_{i=1}^{m} \textrm{Huber}_{\gamma}(\|\q_{i}-a\X\q_{i}\|), \quad 
\textrm{Huber}_{\gamma}(x):=
\begin{cases}
\frac{1}{2}x^{2} &\text{if $|x|\leq \gamma$}\\
\gamma (|x|-\frac{1}{2}\gamma) &\text{else} .
\end{cases}
\end{align*}
Here $a\in(0,1]$ is a regularization parameter and we set it to slightly less than one.

The second model we consider is that of sparsely corrupted entries in which  we again draw a uniformly distributed projection matrix $\P$. This time the data points $\q_1,\dots,\q_m$ are generated by taking $\q_i = \P\z_i/\Vert{\P\z_i}\Vert$ for each $i\in[m]$, where as before $\z_1,\dots,\z_m$ are i.i.d. uniformly distributed unit vectors, but for each $i\in[m]$, with probability $p$, we pick a uniformly distributed entry $j\in[n]$ and set it to $-1$ or $+1$ (with equal probability). The goal is to recover $\P$ by minimizing the following objective function over $\mF_{n,k}$:
\begin{align*}
f(\X)=\sum_{i=1}^{m}\sum_{j=1}^{n} \textrm{Huber}_{\gamma}([\q_{i}]_{j}-[a\X\q_{i}]_{j}),
\end{align*}
where here also $a\in(0,1]$ is a regularization parameter.

For both models we set the Huber loss parameter to $\gamma = 0.1$. For the first model we set $a=0.9$ and for the second $a = 0.8$. For a given projection matrix $\X\in\mP_{n,k}$, we measure the recovery error according to $\Vert{\X-\P}\Vert_F^2$. For both models we let $\X_{PCA}\in\mP_{n,k}$ denote the standard PCA solution, i.e., the projection matrix onto the span of the top $k$ eigenvectors of the empirical covariance $\frac{1}{m}\sum_{i=1}^m\q_i\q_i^{\top}$. For both models we set $n=100$, $k=10$, and $m=500$. For both models we use the projected gradient method to find a projection matrix $\X^*\in\mP_{n,k}$ which has negligible dual gap ($<10^{-10}$). \footnote{For $\X\in\mF_{n,k}$ the dual gap is defined as $\textrm{dg}(\X) = \langle{\X-\V,\nabla{}f(\X)}\rangle$, were $\V\in\argmin_{\Z\in\mP_{n,k}}\langle{\Z,\nabla{}f(\X)}\rangle$. Since $f(\cdot)$ is convex, we in particular have $f(\X) - \min_{\Y\in\mF_{n,k}}f(\Y) \leq \textrm{dg}(\X)$.} For this $\X^*$ we measure the corresponding eigen-gap $\lambda_{n-k}(\nabla{}f(\X^*)) - \lambda_{n-k+1}(\nabla{}f(\X^*))$, and the recovery error $\Vert{\X^*-\P}\Vert_F$. The results are given in Table \ref{table:1}. For each set of parameters the results are the average of 20 i.i.d. experiments.

\begin{figure}
	\centering
	\begin{minipage}[b]{0.24\textwidth}
		\includegraphics[width=\textwidth]{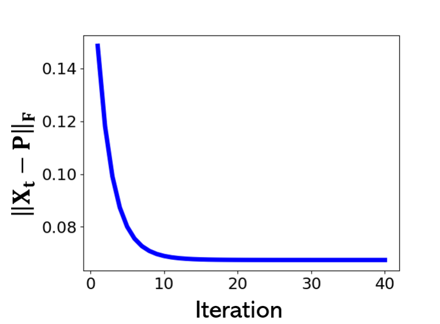}
	\end{minipage}
	\begin{minipage}[b]{0.24\textwidth}
		\includegraphics[width=\textwidth]{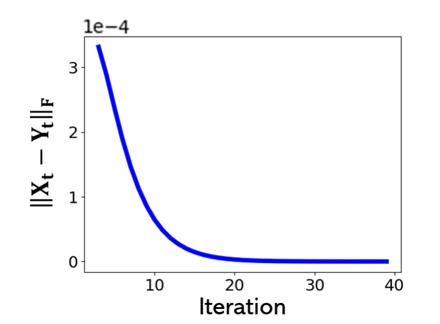}
	\end{minipage}
	\begin{minipage}[b]{0.24\textwidth}
		\includegraphics[width=\textwidth]{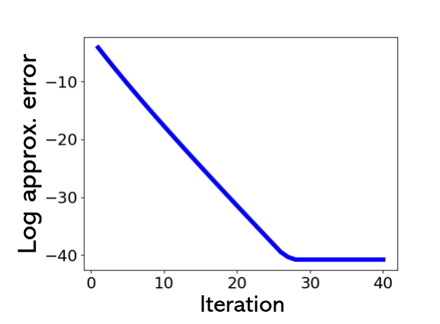}
	\end{minipage}
	\begin{minipage}[b]{0.24\textwidth}
		\includegraphics[width=\textwidth]{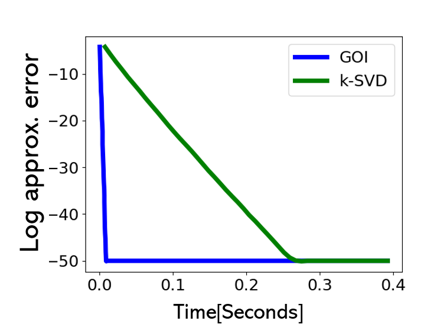}
	\end{minipage}
  \caption{Convergence of PGD and GOI for the sparsely corrupted entries model with $p=0.1$. 1st and 3rd panels from the left show the recovery error ($\P$ is the ground truth projection matrix) and approximation error w.r.t. objective value of PGD, respectively. Convergence of GOI is omitted since it coincides with that of PGD. 2nd panel from the left shows the distance (in Frobenius norm) between the iterates of PGD $(\X_t)$ and those of GOI $(\Y_t)$. The rightmost panel shows the approximation error (in log scale) vs. time, when only the time to compute matrix factorizations is taken into account.} \label{fig:model2}
\end{figure}


In a second experiment we fix for both models $p=0.1$ and vary the dimension $n$ (while keeping $k,m$ fixed as before). The results are given in Table \ref{table:2}. In particular, we see that the eigen-gap $\delta$ does not change substantially with the dimension.

\begin{table*}\renewcommand{\arraystretch}{1.3}
{\footnotesize
\begin{center}
\begin{tabular}{ | P{8em} | P{5em} | P{5em}| P{4em} | P{4em}  |} 
  \hline
  dim. ($n$) &  100 & 200 & 300 & 400  \\ 
  \hline
\multicolumn{5}{|c|}{$\downarrow$ Model 1: spiked covariance $\downarrow$ }\\ \hline
  Eigen-gap ($\delta$) &2.87& 3.02 & 2.96 & 3.04 \\
  \hline
  $\|\X^{*}-\P\|_{F}$ & $0.0071$ & $0.005$ & $0.0043$ & $0.0035$  \\ 
  \hline
  $\|\X_{PCA}-\P\|_{F}$ & 0.068 & 0.049& 0.043&  0.036  \\ 
  \hline
  \multicolumn{5}{|c|}{$\downarrow$ Model 2: sparsely corrupted entries $\downarrow$ }\\ \hline
  Eigen-gap ($\delta$) & 5.49 & 5.902 & 6.06 & 6.1 \\
  \hline
  $\|\X^{*}-\P\|_{F}$ & $0.067$ & $0.0617$ & $0.058$ & $0.055$  \\ 
  \hline
  $\|\X_{PCA}-\P\|_{F}$ & 0.199 & 0.208& 0.206&  0.202  \\ 
  \hline
\end{tabular}
\caption{Recovery and eigen-gap results for the spiked covariance model and corrupted entries model with varying dimension. Each result is the average of 20 i.i.d. experiments.}\label{table:2}
\end{center}
}
\vskip -0.2in
\end{table*}\renewcommand{\arraystretch}{1} 

We turn to demonstrate the empirical performance of the projected gradient method w.r.t. to the nonconvex set $\mP_{n,k}$ (PGD), as given in Dynamics \eqref{eq:nonconvProjGrad}, and and gradient orthogonal iteration (GOI), as given in Dynamics \eqref{eq:QRgrad}, for the two models discussed above. We fix $n=100$ and $p=0.1$ (keeping all other parameters unchanged). For both methods we  use the fixed step-size $\eta = 1/\lambda$, where $\lambda = \lambda_1(\sum_{i=1}^m\q_i\q_i^{\top})$, i.e., the largest eigenvalue of the (unnormalized) empirical covariance. We note that smaller values of $\eta$ seem too conservative in practice from our experimentations. We initialize both methods with the $k$-PCA projection matrix $\X_{PCA}$.  We examine the convergence of both methods in terms of recovery error and approximation error (w.r.t. the objective function). Additionally, to showcase the benefit of avoiding exact SVD computations (as employed by nonconvex PGD) and using only a single QR factorization per iteration (as in GOI), we compare the runtimes of GOI and nonconvex PGD, but we exclude the time it takes to compute the gradient on each iteration and only account for the time it takes to perform either a $\rank-k$ SVD or a QR factorization, where both algorithms were implemented in Python and we have used the built-in functions \textsc{numpy.linalg.eigh} and \textsc{numpy.linalg.qr} to compute thin-SVDs and QR factorizations, respectively. 

The results for the spiked covariance model are given in Figure \ref{fig:model1}, and the results for the sparsely corrupted entries model, which are very similar, are given in Figure \ref{fig:model2}. It can be seen that indeed the distance between the iterates of the two methods decays very quickly and so the graphs of the recovery and approximation errors of both methods coincide. We in particular see that both methods indeed demonstrate a linear convergence rate (w.r.t. the objective value). We also see the significant savings in runtime when replacing a thin-SVD computation (as used by nonconvex PGD) with only a single QR factorization (as used by GOI).

Moreover, in order to verify the convergence of nonconvex PGD to the global optimal solution (and not just a stationary point of the nonconvex Problem \eqref{eq:optProbNonconv}), we verify using the procedure suggested in Remark \ref{remark:rank}, that on each iteration $t$, the projection step onto the Fantope $\mF_{n,k}$ is also of rank $k$, i.e., identical to the projection onto $\mP_{n,k}$. This means that the iterates of PGD w.r.t. the nonconvex Problem \eqref{eq:optProbNonconv} and the iterates of PGD w.r.t. the convex relaxation \eqref{eq:optProbConv}, coincide. Indeed, for all random instances generated and for all iterations executed, we observe that the projection onto the Fantope is of rank $k$. This suggests that the nonconvex PGD (and consequently also GOI) in particular converges to the global optimal solution of the convex relaxation \eqref{eq:optProbConv}.

\section{Proof of Theorem \ref{thm:strictComp}}

\begin{proof}
First, observe that for any dual solution $(\Z_{1}^{*}, \Z_{2}^{*}, s^{*})$, it holds that $\Z_{1}^{*}$ and $\Z_{2}^{*}$ are orthogonal to each other. This is true since, denoting by $\X^*$ the corresponding primal solution, we have that,
\begin{align*}
\langle \Z_{1}^{*}, \Z_{2}^{*} \rangle &= \trace(\Z_{1}^{*}\Z_{2}^{*}) = \trace(\Z_{1}^{*}(\X^{*}+(\I-\X^{*}))\Z_{2}^{*}) \\
&=\trace(\Z_{1}^{*}\X^{*}\Z_{2}^{*}) + \trace(\Z_{1}^{*}(\I-\X^{*})\Z_{2}^{*}) = 0,
\end{align*}
where the last equality follows from the complementarity conditions $\Z_{1}^{*}\X^{*}=\mathbf{0}$ and $(\I-\X^{*})\Z_{2}^{*}=\mathbf{0}$. 

For a given dual solution $(\Z_{1}^{*},\Z_{2}^{*},s^{*})$, let us denote $r_{1}=\rank(\Z_{1}^{*})$ and $r_{2}=\rank(\Z_{2}^{*})$.

Let us write the eigen decompositions $\Z_{1}^{*}=\sum_{i=1}^{r_{1}} \rho_{i}\u_{i}\u_{i}^{T}$ and $\Z_{2}^{*}=\sum_{j=1}^{r_{2}} \mu_{j}\v_{j}\v_{j}^{T}$.

From the orthogonality of $\Z_{1}^{*}$ and $\Z_{2}^{*}$ established above, we get an orthonormal set of vectors $\{\u_{1},...,\u_{r_{1}}, \v_{1},...,\v_{r_{2}}\}$ and we can complete it to an orthonormal basis of $\mathbb{R}^{n}$: 
\begin{align*}
B=\{\u_{1},...,\u_{r_{1}}, \v_{1},...,\v_{r_{2}}, \w_{1},...,\w_{n-r_{1}-r_{2}}\},
\end{align*}
where $\Z_{1}^{*}\w_{i}=\mathbf{0}$ and $\Z_{2}^{*}\w_{i}=\mathbf{0}$ for any $i\in\{1,\dots,n-r_1-r_2\}$.

From the KKT conditions for Problem \eqref{eq:optProbConv}, we have that $\nabla f(\X^{*})= \Z_{1}^{*}-\Z_{2}^{*}+s^{*}\I$, and so it follows that any $\v \in B$ is an eigenvector of $\nabla f(\X^{*})$. Thus, we can write the eigenvalues of $\nabla f(\X^{*})$ in non-increasing order from left to right as:
\begin{align}\label{eq:strictCompProof:1}
\rho_{1}+s^{*},...,\rho_{r_{1}}+s^{*},\underbrace{s^{*},...,s^{*}}_{n-r_{1}-r_{2}~ \textrm{times}},s^{*}-\mu_{r_2},...,s^{*}-\mu_{1}.
\end{align}

For the first direction of the theorem, let us assume $\X^*$ satisfies strict complementarity, so for some dual solution $(\Z_{1}^{*}, \Z_{2}^{*}, s^{*})$ we have that $r_{1}=n-k$ or $r_{2}=k$.

Now, if $r_{1}=n-k$, using \eqref{eq:strictCompProof:1} we get that $\lambda_{n-k}(\nabla f(\X^{*}))=\rho_{n-k}+s^{*}$ and
\newline $\lambda_{n-k+1}(\nabla f(\X^{*})) \leq s^{*}$, and so there is a gap of
\begin{align*}
\lambda_{n-k}(\nabla f(\X^{*}))-\lambda_{n-k+1}(\nabla f(\X^{*})) \geq \rho_{n-k}+s^{*} - s^{*} = \rho_{n-k}>0.
\end{align*}
Otherwise, if $r_{2}=k$, then using \eqref{eq:strictCompProof:1} we have that $\lambda_{n-k+1}(\nabla f(\X^{*})) = s^{*}-\mu_{k}$ and
\newline $\lambda_{n-k}(\nabla f(\X^{*})) >s^{*}$, and so there is a gap of 
\begin{align*}
\lambda_{n-k}(\nabla f(\X^{*}))-\lambda_{n-k+1}(\nabla f(\X^{*})) \geq s^{*} - (s^{*} - \mu_{k}) =\mu_{k} >0.
\end{align*}
In both cases we get a positive eigen-gap, which proves the first direction of the theorem.

For the reversed direction, let us assume that $\X^*$ satisfies the eigen-gap assumption, and recall that according to Theorem \ref{thm:optSol} it follows that $\rank(\X^*) = k$.  Suppose by way of contradiction that there exists a dual solution  $(\Z_1^*,\Z_2^*,s^*)$ for which $r_{1} < n-k$ and $r_{2} < k$. In this case we have from \eqref{eq:strictCompProof:1} that,
\begin{align*}
\lambda_{n-k}(\nabla f(\X^{*})) = \lambda_{n-k+1}(\nabla f(\X^{*})) = s^{*},
\end{align*}
which contradicts the existence of an eigen-gap and  so, it must be that $r_{1}=n-k$ or $r_{2}=k$.
\end{proof}

\section{Details Missing from Section \ref{sec:analPrelim} and Proof of Theorem \ref{thm:optSol}}

\subsection{Proof of Lemma \ref{lem:proj}}
\begin{proof}
The first part of the lemma is a known fact, see for instance \cite{vu2013fantope}.
For the second part, let us prove that if $\sum_{i=1}^r \min(\gamma_i-\gamma_{r+1},1)\geq k$, then $\theta$ must satisfy $\theta\geq\gamma_{r+1}$. Assume by way contradiction that $\theta<\gamma_{r+1}$. Then, 
\begin{align*}
k=\sum_{i=1}^n \min(\max(\gamma_i-\theta,0),1)>\sum_{i=1}^n \min(\max(\gamma_i-\gamma_{r+1},0),1)=\sum_{i=1}^r \min(\gamma_i-\gamma_{r+1},1),
\end{align*}
which is a contradiction, and so it must be that $\theta\geq\gamma_{r+1}$, and in that case the projection sets all the bottom $n-r$ components of the eigen-decomposition of $\X$ to zero. Hence, $\rank(\Pi_{\mF_{n,k}}[\X])\leq r$.
The reversed direction holds from similar reasoning.
\end{proof}

\subsection{Proof of Theorem \ref{thm:optSol}}
Before we prove Theorem \ref{thm:optSol} we need 
the following lemma which is central to our analysis and connects between an optimal solution and the eigen-decomposition of its corresponding gradient.
\begin{lemma}\label{lem:opt}
Let $\X^{*}\in \mF_{n,k}$ be an optimal solution to Problem \eqref{eq:optProbConv} and write the eigen-
\newline decomposition of $-\nabla f(\X^{*})$ as $-\nabla f(\X^{*})=\sum_{i=1}^{n} \mu_{i} \u_{i}\u_{i}^{\top}$.
Let $r$ be the smallest integer such that $r\geq k$ and $\mu_{k}-\mu_{r+1}>0$. Then,  for all $n \geq i \geq r+1$, $\X^{*}$ is orthogonal to $\u_i\u_i^{\top}$, and $\rank(\X^{*})\leq r$. 
\newline In particular, if $r=k$, then $\X^{*}\in\mP_{n,k}$ is the unique projection matrix onto the span of the $k$ leading eigenvectors of  $-\nabla f(\X^{*})$.
\end{lemma}

\begin{proof}
Assume by way of contradiction that $\X^{*}$ is not orthogonal $\u_{r+1}\u_{r+1}^{\top},\dots,\u_n\u_n^{\top}$. In this case, $\sum_{i=r+1}^{n} \u_{i}^{\top} \X^{*} \u_{i} >0$, and we can write,
\begin{align*}
 \langle \X^{*}, -\nabla f(\X^{*}) \rangle &= \sum_{i=1}^{r} \mu_i \u_i^{\top} \X^* \u_i +  \sum_{i=r+1}^{n} \mu_i \u_i^{\top} \X^* \u_i \nonumber \\
 & \underset{(a)}{<}  \sum_{i=1}^{r} \mu_{i}\u_i^{\top} \X^* \u_i +  \mu_{r}\sum_{i=r+1}^{n} \u_i^{\top} \X^* \u_i \nonumber\\
 & \underset{(b)}{=} \sum_{i=1}^{k-1}\mu_i\u_i^{\top}\X^*\u_i + \mu_k\sum_{i=k}^n\u_i^{\top}\X^*\u_i \nonumber \\
 & \underset{(c)}{=} \sum_{i=1}^{k-1}\mu_i\u_i^{\top}\X^*\u_i + \mu_k\left({k - \sum_{i=1}^{k-1}\u_i^{\top}\X^*\u_i}\right),
\end{align*}
where both (a) and (b) follow from the definition of $r$, and (c) follows since $\sum_{i=1}^n\u_i^{\top}\X^*\u_i = \trace(\X^*\sum_{i=1}^n\u_i\u_i^*) = \trace(\X^*\I)=k$.

Let us denote the projection matrix onto the span of the top $k$ eigenvectors of $-\nabla{}f(\X^*)$ by $\P^* = \sum_{i=1}^k\u_i\u_i^{\top}$, and note that $\langle{\P^*,-\nabla{}f(\X^*)}\rangle = \sum_{i=1}^k\mu_i$. It follows that
\begin{align*}
\langle{\P^*-\X^*,\nabla{}f(\X^*)}\rangle &= \langle{\X^*-\P^*,-\nabla{}f(\X^*)}\rangle \\
&<  \sum_{i=1}^{k-1}\mu_i\u_i^{\top}\X^*\u_i + \mu_k\left({k - \sum_{i=1}^{k-1}\u_i^{\top}\X^*\u_i}\right) - \sum_{i=1}^k\mu_i \\
& = \sum_{i=1}^{k-1}\mu_i\left({\u_i^{\top}\X^*\u_i -1}\right)+ \mu_k\sum_{i=1}^{k-1}\left({1 - \u_i^{\top}\X^*\u_i}\right) \\
& = \sum_{i=1}^{k-1}\left(1 - {\u_i^{\top}\X^*\u_i}\right)(\mu_k - \mu_i) \leq 0, 
\end{align*}
where the last inequality follows since for all $i$, $\u_i^{\top}\X^*\u_i\in[0,1]$.

Thus, we have that $\X^*$ violates the first-order optimality condition which contradicts that assumption that it is an optimal solution, and thus we have that $\X^*$ must indeed be orthogonal to $\u_{r+1}\u_{r+1}^{\top},\dots,\u_n\u_n^{\top}$. 

An immediate consequence is that  the eigenvectors of $\X^{*}$ which correspond to non-zero eigenvalues must lie in  $\textrm{span}\{\u_{1},...,\u_{r}\}$ and thus, it must be that $\rank(\X^{*})\leq r$. 

For the final part of the lemma, in case $r=k$, since for all $\X\in\mF_{n,k}$, $\rank(\X) \geq k$, we have that $\rank(\X^*)=k$. In particular,  $\X^*$ is a projection matrix, i.e., $\X\in\mP_{n,k}$. By the orthogonality result above, it follows that the eigenvectors of $\X^*$ lie in $\textrm{span}\{\u_1,\dots,\u_k\}$, which means that $\X^*$ is indeed the projection matrix onto  $\textrm{span}\{\u_1,\dots,\u_k\}$, as stated in the lemma. Note that when $r=k$, this projection matrix is indeed unique (i.e., the subspace spanned by the top $k$ eigenvectors of $-\nabla{}f(\X^*)$ is unique).
\end{proof} 

We now prove Theorem \ref{thm:optSol}.

\begin{proof}[Proof of Theorem \ref{thm:optSol}]
Let $\X^*$ be an optimal solution to the convex relaxation \eqref{eq:optProbConv} which satisfies Assumption \ref{ass:gap} with some $\delta >0$. It follows directly from Lemma \ref{lem:opt} that $\rank(\X^*) = k$. From Lemma \ref{lem:opt} it further follows that $\X^*$ is the unique projection matrix onto the span of top $k$ eigenvectors of $-\nabla{}f(\X^*)$, i.e., it is the unique matrix in $\X\in\mP_{n,k}$ such that $\langle{\X,-\nabla{}f(\X^*)}\rangle = \sum_{i=1}^k\mu_i$, where we write the eigen-decomposition of $-\nabla{}f(\X^*)$ as $-\nabla{}f(\X^*) = \sum_{i=1}^n\mu_i\u_i\u_i^{\top}$. From the von Neumann trace inequality it follows that for any matrix $\X\in\mP_{n,k}$ it holds that $\langle{\X,-\nabla{}f(\X^*)}\rangle \leq \sum_{i=1}^n\lambda_i(\X)\mu_i = \sum_{i=1}^k\lambda_i(\X)\mu_i = \sum_{i=1}^k\mu_i$. Thus, we have that
for all $\X\in\mP_{n,k}\setminus\{\X^*\}:$  $\langle{\X-\X^*,\nabla{}f(\X^*)}\rangle = \langle{\X^*-\X,-\nabla{}f(\X^*)}\rangle > 0$.
Since $\mF_{n,k} = \conv\{\mP_{n,k}\}$, this further implies that
for all $\X\in\mF_{n,k}\setminus\{\X^*\}$: $\langle{\X-\X^*,\nabla{}f(\X^*)}\rangle > 0$.
Since $f(\cdot)$ is convex, it further holds that
for all $\X\in\mF_{n,k}\setminus\{\X^*\}$: $f(\X^*)- f(\X) \leq \langle{\X^*-\X,\nabla{}f(\X^*)}\rangle < 0$,
and thus, we conclude that $\X^*$ is indeed the unique optimal solution to Problem \eqref{eq:optProbConv}, which also implies that it is the unique optimal solution to Problem \eqref{eq:optProbNonconv}.
\end{proof}

\subsection{Proof of Lemma \ref{lem:quad}} 
\begin{proof}

Let us write the eigen-decomposition of the gradient $\nabla f(\X^{*})$ as  $\nabla f(\X^{*})=\sum_{i=1}^{n} \lambda_{i} \u_{i} \u_{i}^{\top}$. For any $\X\in \mF_{n,k}$ it holds that:
\begin{align}\label{quad:eq:1}
   f(\X)-f(\X^{*}) &\underset{(a)}{\geq} \langle \X-\X^{*},\nabla f(\X^{*}) \rangle \underset{(b)}{=} \sum_{i=1}^{n} \lambda_{i}  \u_{i}^{\top} \X \u_{i} - \sum_{i=n-k+1}^{n} \lambda_{i} \nonumber \\
  &\underset{(c)}{\geq} (\lambda_{n-k+1} +\delta)  \sum_{i=1}^{n-k}  \u_{i}^{\top} \X \u_{i} +  \sum_{i=n-k+1}^{n} \lambda_{i}  \u_{i}^{\top} \X \u_{i} -  \sum_{i=n-k+1}^{n} \lambda_{i}  \nonumber \\
  &=(\lambda_{n-k+1} +\delta)  \sum_{i=1}^{n-k}  \u_{i}^{\top} \X \u_{i} -  \sum_{i=n-k+1}^{n} \lambda_{i}  (1-\u_{i}^{\top} \X \u_{i})  \nonumber \\
  &\underset{(d)}{\geq}  (\lambda_{n-k+1} +\delta)  \sum_{i=1}^{n-k}  \u_{i}^{\top} \X \u_{i} -  \lambda_{n-k+1}\sum_{i=n-k+1}^{n} (1-\u_{i}^{\top} \X \u_{i})  \nonumber \\
  &= \lambda_{n-k+1} \sum_{i=1}^{n} \u_{i}^{\top}\X\u_{i} -k\lambda_{n-k+1} +\delta\sum_{i=1}^{n-k} \u_{i}^{\top}\X\u_{i},
\end{align}
where (a) follows from the convexity of $f(\X)$, (b) follows  since according Lemma \ref{lem:opt} $\X^*=\sum_{i=n-k+1}^n\u_i\u_i^{\top}$ and so, $\langle{\X^*,\nabla{}f(\X^*)}\rangle = \sum_{i=n-k+1}^n\lambda_i$, (c) follows from Assumption \ref{ass:gap}, and (d) follows since $\X\preceq\I$, which implies that $\u_{i}^{\top}\X\u_{i}\leq \u_{i}^{\top}\u_{i}=1$.

Using $\sum_{i=1}^{n} \u_{i}^{\top} \X \u_{i} =\trace(\X\sum_{i=1}^n\u_i\u_i^{\top})=\trace(\X\I)=k$ and Eq. \eqref{quad:eq:1}, we have,
\begin{align} \label{quad:eq:2}
 f(\X)-f(\X^{*})\geq\delta\sum_{i=1}^{n-k} \u_{i}^{\top} \X \u_{i} = \delta \left({k-\sum_{i=n-k+1}^{n} \u_{i}^{\top} \X \u_{i}}\right).
\end{align}
Also, using again the fact that $\X^* = \sum_{i=n-k+1}^n\u_i\u_i^{\top}$, we have that, 
\begin{align} \label{quad:eq:3}
\|\X-\X^{*}\|_{F}^{2} = \|\X\|_{F}^{2}+\|\X^{*}\|_{F}^{2}-2\sum_{i=n-k+1}^{n} \u_{i}^{\top} \X \u_{i} \leq 2\left({k-\sum_{i=n-k+1}^{n} \u_{i}^{\top} \X \u_{i} }\right),
\end{align}
where the last inequality follows since for any $\X\in \mF_{n,k}$ it holds that $\|\X\|_{F}^{2}=\sum_{i=1}^{n} \lambda_{i}^{2}(\X) \leq \sum_{i=1}^{n} \lambda_{i}(\X)  =k$.
Combining  Eq. \eqref{quad:eq:2} and  \eqref{quad:eq:3} we finally have that,
\begin{align*} 
f(\X)-f(\X^{*})\geq \dfrac{\delta}{2} \| \X-\X^{*}\| _{F}^{2}.
\end{align*}
\end{proof}

\section{Projected Gradient Descent Analysis}
In this section we turn to analyze the local convergence of the projected gradient method w.r.t. the sets $\mP_{n,k}$ and $\mF_{n,k}$, and to prove Theorems \ref{pgd_theorem} and \ref{thm:pgd_general}.

We first provide the proof of Lemma  \ref{lem:rad} which is fairly simple, and then prove a more general version of the lemma, which in particular allows to relax Assumption \ref{ass:gap}. 

\begin{proof}[Proof of Lemma \ref{lem:rad}]
Denote $\Y^*=\X^*-\eta\nabla f(\X^*)$ and denote the eigenvalues of $\Y^*$ in non-increasing order $\sigma_i=\lambda_{i}(\Y^{*}), i=1,\dots,n$. Denote also  $\Y=\X-\eta\nabla f(\X)$ with its eigenvalues $\gamma_i = \lambda_{i}(\Y), i=1,\dots,n$. Let us write the eigen-decomposition of $-\nabla{}f(\X^*)$ as $-\nabla{}f(\X^*) = \sum_{i=1}^n\mu_i\u_i\u_i^{\top}$. From Lemma \ref{lem:opt}, we have that under Assumption \ref{ass:gap}, it holds that $\X^* = \sum_{i=1}^k\u_i\u_i^{\top}$. Thus, we can deduce that
\begin{align}\label{rad:eq:1}
\sigma_i=
\begin{cases}
1+\eta\mu_i  &\text{if $i\in\{1,...,k\};$}\\
\eta\mu_i  &\text{else.}
\end{cases}
\end{align}
From Lemma \ref{lem:proj} we have that $\rank(\Pi_{\mF_{n,k}}(\Y))=k$ if and only if $\sum_{i=1}^k \min(\gamma_i-\gamma_{k+1},1)\geq k$. Thus, a sufficient condition so that  $\rank(\Pi_{\mF_{n,k}}(\Y))=k$ is,
\begin{align}\label{rad:eq:2}
\gamma_k-\gamma_{k+1}\geq 1.
\end{align}
By  Weyl's inequality for the eigenvalues and Eq. \eqref{rad:eq:1} we have,
\begin{align*}
  \gamma_k-\gamma_{k+1}&=(\sigma_k-\sigma_{k+1}) + (\gamma_k-\sigma_k)+(\sigma_{k+1}-\gamma_{k+1})\\
  &\geq  1+\eta(\mu_k-\mu_{k+1})-2\|\Y-\Y^*\|_F   \\
  &=1+\eta(\mu_k-\mu_{k+1})-2\|\X-\X^*-\eta\nabla f(\X)+\eta\nabla f(\X^*)\|_F  \\
  &\geq 1+\eta(\mu_k-\mu_{k+1})-2(1+\eta\beta)\|\X-\X^*\|_F.
\end{align*}
Thus, we see that a sufficient condition so that \eqref{rad:eq:2} holds, is that $\X$ satisfies
\begin{align*}
\|\X-\X^*\|_F \leq \dfrac{\eta\delta}{2(1+\eta\beta)} \leq \dfrac{\eta(\mu_k-\mu_{k+1})}{2(1+\eta\beta)},
\end{align*}
and so the lemma follows.
\end{proof}

The following lemma generalizes Lemma \ref{lem:rad} and offers a natural trade-off between the rank of the projected gradient mapping and the size of the ball around an optimal solution $\X^*$ in which it is guaranteed to be upper-bounded. 

\begin{lemma} \label{lem:genrad}
Let $\X^{*}\in \mF_{n,k}$ be an optimal solution to Problem \eqref{eq:optProbConv}, and let $\mu_1 \geq \mu_2 \geq \dots \mu_n$ denote the eigenvalues of  $-\nabla f(\X^{*})$.
Let $r$ be the smallest integer such that $r\geq k$ and $\mu_{k}>\mu_{r+1}$. Fix some $\eta >0$. For any $\X\in \mF_{n,k}$ which satisfies
\begin{align}\label{genrad:eq:40}
\|\X-\X^*\|_F\leq\dfrac{\eta(\mu_k-\mu_{r+1})}{2(1+\eta\beta)},
\end{align}
it holds that $\rank(\Pi_{\mF_k}(\X-\eta\nabla f(\X)))\leq r$.
\newline More generally, for any $r'\in\{r,r+1,...,n-1\}$ and for any $\eta>0$, if $\X\in \mF_{n,k}$ satisfies
\begin{align}\label{genrad:eq:41}
\|\X-\X^*\|_F\leq\dfrac{\eta(\mu_k-\mu_{r'+1})}{2(1+\eta\beta)},
\end{align}
then $\rank(\Pi_{\mF_k}(\X-\eta\nabla f(\X)))\leq r'$.
\end{lemma}

\begin{proof}
From Lemma \ref{lem:opt} we have that $r^{*}:=\rank(\X^{*})\leq r$.
Denote $\Y^*=\X^*-\eta\nabla f(\X^*)$ and denote the eigenvalues of $\Y^*$ as $\sigma_i = \lambda_{i}(\Y^{*}), i=1,\dots,n$. Denote also $\Y=\X-\eta\nabla f(\X)$ with its eigenvalues $\gamma_i = \lambda_{i}(\Y), i=1,\dots,n$.
\newline From the min-max principle for the eigenvalues, letting $\mV\subseteq\reals^n$ denote some subspace of $\reals^n$,  we have that for any $i \in \{1,...,r\}$,
\begin{align}\label{genrad:eq:1}
\sigma_{i}=\min_{\mV:\dim(\mV)=n-i+1}\max_{\v\in \mV:\|\v\|=1}\v^{\top}(\X^*+\eta(-\nabla f(\X^*)))\v.
\end{align}
Let us write the eigen-decomposition of $-\nabla{}f(\X^*)$ as $-\nabla{}f(\X^*) = \sum_{i=1}^n\mu_i\u_i\u_i^{\top}$.  Note that in Eq. \eqref{genrad:eq:1} we minimize over all the subspaces $\mV$ of dimension $n-i+1$, $i \leq r$, and so,
\begin{align}\label{genrad:eq:2}
\mV\cap \textrm{span}\{\u_{1},...,\u_{r}\} \neq \emptyset,
\end{align}
otherwise the direct sum $\mV\oplus \textrm{span}\{\u_{1},...,\u_{r}\}\subseteq \reals^{n}$ would have dimension $n-i+1+r>n$.

Any unit vector $\v\in \mV$ can be written as $\v=a\u+b\w$ such that $\u\in \textrm{span}\{\u_{1},...,\u_{r}\}$, $\Vert{\u}\Vert=1$, $\w\in \textrm{span}\{\u_{r+1},...,\u_{n}\}, \Vert{\w}\Vert=1$, and $a^2+b^2=1$. Thus, for any such unit vector $\v$, using Lemma \ref{lem:opt}, we have that,
\begin{align}\label{genrad:eq:3}
\v^{\top}(\X^*+\eta(-\nabla f(\X^*)))\v = a^{2} \u^{\top}\X^{*}\u + a^{2}\eta \u^{\top}(-\nabla f(\X^{*}))\u + b^{2} \eta \w^{\top}(-\nabla f(\X^{*}))\w.
\end{align}
Note that
\begin{align}\label{genrad:eq:4}
\u^{\top}(-\nabla f(\X^{*}))\u  \geq \mu_{r} > \mu_{r+1} \geq \w^{\top}(-\nabla f(\X^{*}))\w.
\end{align}
This implies that the inner maximum in \eqref{genrad:eq:1} can only be obtained by vectors in $\mV\cap\textrm{span}\{\u_{1},...,\u_{r}\}$ (note \eqref{genrad:eq:2} guarantees such vectors exist). Thus, plugging this observation into  \eqref{genrad:eq:1} we have that for any $i\in\{1,\dots,r\}$, 
\begin{align}\label{genrad:eq:5}
\sigma_{i}&=\min_{\mV:\dim(\mV)=n-i+1}\max_{\v\in \mV\cap\textrm{span}\{\u_{1},...,\u_{r}\},\|\v\|=1}\v^{\top}(\X^*+\eta(-\nabla f(\X^*)))\v   \nonumber \\
&{\geq}\min_{\mV:\dim(\mV)=n-i+1}\max_{\v\in \mV\cap\textrm{span}\{\u_{1},...,\u_{r}\},\|\v\|=1}\v^{\top}\X^{*}\v +\eta\mu_{r} \nonumber \\
&\underset{(a)}{=}\min_{\mV:\dim(\mV)=n-i+1}\max_{\v\in \mV,\|\v\|=1}\v^{\top}\X^{*}\v +\eta\mu_{r}   \underset{(b)}{=} \lambda_{i}(\X^{*})+\eta\mu_{k},
\end{align}
where (a) follows from the orthogonality of $\X^{*}$ to $\u_{r+1}\u_{r+1}^{\top},...,\u_{n}\u_{n}^{\top}$ (see Lemma \ref{lem:opt}), and (b) follows from the min-max principle for the eigenvalues, and since by definition $\mu_r = \mu_k$.

Using the max-min principle for the eigenvalues, we can write for any $j \in \{r+1,...,n\}$,
\begin{align}\label{genrad:eq:6}
\sigma_{j}=\max_{\mV:\dim(\mV)=j}\min_{\v\in \mV:\|\v\|=1}\v^{\top}(\X^*+\eta(-\nabla f(\X^*)))\v.
\end{align}
This time we maximize over all subspaces of dimension $j$, $j\geq r+1$. Thus, it must hold that for each such subspace $\mV$,
\begin{align*}
\mV\cap \textrm{span}\{\u_{r+1},...,\u_{n}\} \neq \emptyset,
\end{align*}
otherwise the direct sum $\mV\oplus \textrm{span}\{\u_{r+1},...,\u_{n}\}\subseteq \reals^{n}$ would have dimension $j+n-r>n$.
Thus, using \eqref{genrad:eq:3} and \eqref{genrad:eq:4}, we have that the inner minimum in \eqref{genrad:eq:6} is obtained by vectors in $\mV\cap\textrm{span}\{\u_{r+1},\dots,\u_n\}$, which is not an empty set. Using this observation we have that for any $j\in\{r+1,\dots,n\}$,
\begin{align}\label{genrad:eq:7}
\sigma_{j}&=\max_{\mV:\dim(\mV)=j}\min_{\v\in \mV\cap\textrm{span} \{\u_{r+1},...,\u_{n}\},\|\v\|=1}\v^{\top}(\X^*+\eta(-\nabla f(\X^*)))\v \nonumber \\
&\underset{(a)}{=}\max_{\mV:\dim(\mV)=j}\min_{\v\in \mV\cap\textrm{span} \{\u_{r+1},...,\u_{n}\},\|\v\|=1}\v^{\top}(\eta(-\nabla f(\X^*)))\v \nonumber \\
&\underset{(b)}=\max_{\mV:\dim(\mV)=j}\min_{\v\in \mV,\|\v\|=1}\v^{\top}(\eta(-\nabla f(\X^*)))\v = \eta\mu_{j},
\end{align}
where (a) follows since $\X^*$ is orthogonal to $\u_{r+1}\u_{r+1}^{\top},\dots,\u_n\u_n^{\top}$ (see Lemma \ref{lem:opt}),
\newline  and (b) follows since by the eigen-decomposition of $-\nabla{}f(\X^*)$, restricting $\v$ to the intersection $\mV\cap\textrm{span} \{\u_{r+1},...,\u_{n}\}$ does not increase the inner minimum.

From Lemma \ref{lem:proj} we have the sufficient condition so that $\rank(\Pi_{\mF_{n,k}}(\Y))\leq r$: 
\begin{align}\label{genrad:eq:7a}
\sum_{i=1}^r \min(\gamma_i-\gamma_{r+1},1)\geq k ~\Longrightarrow~ \rank(\Pi_{\mF_{n,k}}(\Y))\leq r. 
\end{align}

By  Weyl's inequality we have that for any $i\in\{1,\dots,r\}$,
\begin{align}\label{genrad:eq:8}
\gamma_i-\gamma_{r+1}&\geq\sigma_i-\sigma_{r+1}-2\Vert{\Y-\Y^*}\Vert_F \nonumber \\ 
&=\sigma_i-\sigma_{r+1}-2\Vert{\X-\eta\nabla{}f(\X) - \X^*+\eta\nabla{}f(\X^*)}\Vert_F \nonumber \\
&\geq\sigma_i-\sigma_{r+1}-2(1+\eta\beta)\|\X-\X^*\|_F. 
\end{align}

Thus, we have that
\begin{align}\label{genrad:eq:9}
    \sum_{i=1}^{r} \min(\gamma_{i}-\gamma_{r+1},1)&\underset{(a)}{\geq} \sum_{i=1}^{r} \min(\sigma_i-\sigma_{r+1}-2(1+\eta\beta)\|\X-\X^*\|_F, 1)\nonumber \\
    &\underset{(b)}{\geq}\sum_{i=1}^{r} \min(\lambda_{i}(\X^*)+\eta(\mu_i-\mu_{r+1})-2(1+\eta\beta)\|\X-\X^*\|_F,1) \nonumber \\
    &\geq\sum_{i=1}^{r} \min(\lambda_{i}(\X^*)+\eta(\mu_{r}-\mu_{r+1})-2(1+\eta\beta)\|\X-\X^*\|_F, 1),
\end{align}

where (a) follows from \eqref{genrad:eq:8}, and (b) follows from \eqref{genrad:eq:5} and \eqref{genrad:eq:7}.

Thus, we indeed see that if 
\begin{align*}
\|\X-\X^*\|_F \leq \frac{\eta(\mu_{r}-\mu_{r+1})}{2(1+\eta\beta)},
\end{align*}
then $\sum_{i=1}^{r} \min(\gamma_{i}-\gamma_{r+1},1)\geq\sum_{i=1}^{r^*} \lambda_{i}(\X^*) =k$, which by \eqref{genrad:eq:7a} implies that  $\rank(\Pi_{\mF_{n,k}}(\Y))\leq r$, as needed.

For the second part of the lemma let us fix some $r' \in \{r,...,n-1\}$. 
If we have that
\begin{align}\label{genrad:eq:10}
\|\X-\X^*\|_F \leq \frac{\eta(\mu_{r}-\mu_{r'+1})}{2(1+\eta\beta)},
\end{align} 
then similarly to \eqref{genrad:eq:9}, we will have that,
\begin{align*}
    \sum_{i=1}^{r'} \min(\gamma_{i}-\gamma_{r'+1},1)&\geq\sum_{i=1}^{r} \min(\gamma_{i}-\gamma_{r'+1},1)\\
    &\underset{(a)}{\geq}\sum_{i=1}^{r} \min(\sigma_i - \sigma_{r'+1}-2(1+\eta\beta)\|\X-\X^*\|_F,1) \\
    &\underset{(b)}{\geq}\sum_{i=1}^{r} \min(\lambda_{i}(\X^*)+\eta(\mu_i-\mu_{r'+1})-2(1+\eta\beta)\|\X-\X^*\|_F,1) \\
    &{\geq}\sum_{i=1}^{r} \min(\lambda_{i}(\X^*)+\eta(\mu_r-\mu_{r'+1})-2(1+\eta\beta)\|\X-\X^*\|_F,1)\\
    &\underset{(c)}{\geq} \sum_{i=1}^{r} \min(\lambda_{i}(\X^*),1) = \sum_{i=1}^{r^*} \lambda_{i}(\X^*) =k,
\end{align*}
where (a) follows from the same reasoning as \eqref{genrad:eq:8}, (b) follows from \eqref{genrad:eq:5} and \eqref{genrad:eq:7}, and (c) follows from \eqref{genrad:eq:10}.

Thus, from Lemma \ref{lem:proj} we have that \eqref{genrad:eq:10} indeed implies that $\rank(\Pi_{\mF_{n,k}}(\Y)) \leq r'$, which proves the second part of the lemma.
\end{proof}

We can now easily prove Theorems \ref{pgd_theorem} and \ref{thm:pgd_general} by proving the following unifying theorem.
\begin{theorem}
Let  $\{\X_{t}\}_{t\geq 1}$ be a sequence produced by the projected gradient dynamics w.r.t. the convex Problem \eqref{eq:optProbConv}  with a fixed step-size $\eta\in(0,1/\beta]$:
\begin{align*}
\X_{t+1}=\Pi_{\mF_{n,k}}(\X_{t}-\eta\nabla f(\X_t)).
\end{align*}
Fix some optimal solution $\X^*$ and let $\mu_1 \geq \mu_2 \geq ... \mu_n$ denote the eigenvalues of $-\nabla{}f(\X^*)$. Let $r$ be the smallest integer such that $r\geq k$ and $\mu_{r} - \mu_{r+1} > 0$. 
If the initialization $\X_1\in\mF_{n,k}$ satisfies $\|\X_{1}-\X^*\|_F\leq\dfrac{\eta(\mu_k-\mu_{r+1})}{2(1+\eta\beta)}$, then for all $t\geq 1$, $\rank(\X_{t+1}) \leq r$. 

More generally, for every $r'\in\{r,\dots,n\}$, if $\|\X_{1}-\X^*\|_F\leq\dfrac{\eta(\mu_k-\mu_{r'+1})}{2(1+\eta\beta)}$, then for all $t\geq 1$, $\rank(\X_{t+1}) \leq r'$.

In particular, if $r=k$, i.e., Assumption \ref{ass:gap} holds with some $\delta >0$, and  $\|\X_{1}-\X^*\|_F\leq\dfrac{\delta}{4\beta}$, setting $\eta = 1/\beta$ guarantees that for all $t\geq 1$, $\rank(\X_{t+1}) = k$, and the sequence $\{\X_t\}_{t\geq 1}$ converges linearly with rate:
\begin{align*}
\forall t\geq 1: \quad f(\X_t) - f(\X^*) \leq \left({f(\X_1) - f(\X^*)}\right)\exp\left({-\Theta\left({\delta/\beta}\right)(t-1)}\right).
\end{align*}
\end{theorem}

\begin{proof}
It is a well known fact that the distances of the iterates generated by the projected gradient method with step-size $\eta\in(0,1/\beta]$ to any optimal solution are monotone  non-increasing, i.e., the sequence $\{\Vert{\X_t-\X^*}\Vert_F\}_{t\geq 1}$ is monotone non-increasing, see for instance \cite{bubeck2014convex}. Thus, all results of the theorem regarding the rank of the iterates $\X_t, t\geq 1$, follow immediately from this observation, the initialization conditions listed in the theorem, and Lemma \ref{lem:genrad}.

The linear convergence rate under Assumption \ref{ass:gap} follows from the quadratic growth result --- Lemma \ref{lem:quad}, and the known linear convergence rate of the projected gradient method for smooth functions that satisfy the quadratic growth property, see for instance \cite{necoara2019linear}. 
\end{proof}

\section{Details Missing from Section \ref{sec:GOI}}

We prove two auxiliary lemmas and then prove Theorem \ref{qr_theorem}.

\begin{lemma}\label{lem:aux}
Let $\M\in\mbS^n$, and let $\X\in\mP_{n,k}$ be the projection matrix onto the span of the top $k$ eigenvectors of $\M$. Then, for any $\Z\in\mP_{n,k}$ it holds that,
\begin{align*}
\langle{\X-\Z,\M}\rangle \leq \Vert{\Z-\X}\Vert_F^2\Vert{\M}\Vert_2.
\end{align*}
\end{lemma}
\begin{proof}
Let us denote by $\X_{\perp}$ the projection matrix onto the orthogonal subspace, i.e., $\X_{\perp} = \I - \X$. It holds that,
\begin{align}\label{eq:lem:aux:1}
\langle{\X-\Z,\M}\rangle &= \langle{\X-\Z,\X\M}\rangle
+ \langle{\X-\Z,\X_{\perp}\M}\rangle \nonumber \\
&=\langle{\X-\Z,\X\M}\rangle
- \langle{\Z,\X_{\perp}\M}\rangle.
\end{align}

We consider each of the two terms on the RHS separately.

\begin{align}\label{eq:lem:aux:2}
\langle{\X-\Z,\X\M}\rangle &= \trace((\X-\Z)\X\M) = \trace(\X(\X-\Z)\X\M) \nonumber \\
&\underset{(a)}{\leq} \trace(\X(\X-\Z)\X)\cdot\lambda_{1}(\M) \nonumber\\
&= \langle{\X-\Z,\X}\rangle\cdot\lambda_{1}(\M) = (k-\langle{\Z,\X}\rangle)\cdot\lambda_{1}(\M)  \nonumber\\
&\leq \frac{1}{2}\Vert{\Z-\X}\Vert_F^2\cdot\Vert{\M}\Vert_2,
\end{align}
where  $(a)$ holds since $\X(\X-\Z)\X$ is positive semidefinite.

\begin{align}\label{eq:lem:aux:3}
\langle{\Z,\X_{\perp}\M}\rangle &= \trace(\Z\X_{\perp}\M) = \trace(\X_{\perp}\Z\X_{\perp}\M) \nonumber\\
&\underset{(b)}{\geq} \trace(\X_{\perp}\Z\X_{\perp})\cdot\lambda_n(\M) = \langle{\Z,\X_{\perp}}\rangle\cdot\lambda_n(\M) \nonumber\\
&=\langle{\Z,\I-\X}\rangle\cdot\lambda_n(\M) = (k-\langle{\Z,\X}\rangle)\cdot\lambda_n(\M) \nonumber\\
&\geq -\frac{1}{2}\Vert{\Z-\X^*}\Vert_F^2\cdot\Vert{\M}\Vert_2,
\end{align}
where $(b)$ holds since $\X_{\perp}\Z\X_{\perp}$ is positive semidefinite.

The lemma follows from plugging \eqref{eq:lem:aux:2} and \eqref{eq:lem:aux:3} into \eqref{eq:lem:aux:1}.
\end{proof}

\begin{lemma}\label{lem:distXY}
Fix some $t\geq 1$ and suppose $\eta < 1/\beta$. Then it holds that,
\begin{align*}
\Vert{\X_{t+1}-\Y_t}\Vert_F^2 \leq \frac{\eta}{1-\eta\beta}\left({f(\Y_t) - f(\X_{t+1})}\right).
\end{align*}
\end{lemma}
\begin{proof}
Define the following function
\begin{align*}
\phi(\Z) := \langle{\Z,\nabla{}f(\Y_t)}\rangle + \frac{1}{2\eta}\Vert{\Z-\Y_t}\Vert_F^2,
\end{align*}
and note that it is $1/\eta$ strongly convex, and that by definition, $\X_{t+1}$ is its minimizer over $\mF_{n,k}$. Thus,
\begin{align*}
\Vert{\X_{t+1}-\Y_t}\Vert_F^2 &\leq 2\eta\left({\phi(\Y_t)-\phi(\X_{t+1})}\right) \\
&= 2\eta\langle{\Y_t-\X_{t+1},\nabla{}f(\Y_t)}\rangle -\Vert{\X_{t+1}-\Y_t}\Vert_F^2.
\end{align*}
Rearranging we get,
\begin{align*}
\Vert{\X_{t+1}-\Y_t}\Vert_F^2 &\leq \eta\langle{\Y_t-\X_{t+1},\nabla{}f(\Y_t)}\rangle \nonumber \\
& =  \eta\langle{\Y_t-\X_{t+1},\nabla{}f(\X_{t+1})}\rangle +  \eta\langle{\Y_t-\X_{t+1},\nabla{}f(\Y_t) - \nabla{}f(\X_{t+1})}\rangle
\nonumber \\
&\underset{(a)}{\leq}  \eta\langle{\Y_t-\X_{t+1},\nabla{}f(\X_{t+1})}\rangle +  \eta\beta\Vert{\X_{t+1}-\Y_{t}}\Vert_F^2 \\
&\underset{(b)}{\leq}  \eta\left({f(\Y_t) - f(\X_{t+1})}\right) +  \eta\beta\Vert{\X_{t+1}-\Y_{t}}\Vert_F^2,
\end{align*}
where (a) follows from the $\beta$-smoothness of $f(\cdot)$, and (b) follows from the convexity of $f(\cdot)$.
Rearranging, we get the lemma.
\end{proof}

\begin{proof}[Proof of Theorem \ref{qr_theorem}]
The theorem follows from Lemma \ref{lem:errorDec}, it only remains to prove that the necessary conditions hold for all $t\geq 1$, i.e., that for all $t\geq 1$, it holds that $\Vert{\X_{t+1}-\Y_t}\Vert_F\leq 1$, and $\X_{t+1}\in\mP_{n,k}$, i.e., $\rank(\X_{t+1})=k$.  The proof is by induction. For the base case $t=1$, we first note that using Lemma \ref{lem:distXY} we have that,
\begin{align}\label{eq:qr_thm:1}
\Vert{\X_{2}-\Y_1}\Vert_F^2 &\leq \frac{\eta}{1-\eta\beta}\left({f(\Y_1) - f(\X_{2})}\right) \leq \frac{\eta}{1-\eta\beta}\left({f(\Y_{1})-f(\X^*)}\right)  \nonumber \\
&\underset{(a)}{\leq} \frac{\eta}{1-\eta\beta}\left({\langle{\X^*-\Y_1,-\nabla{}f(\X^*)}\rangle + \frac{\beta}{2}\Vert{\Y_1-\X^*}\Vert_F^2}\right) \nonumber \\
&\underset{(b)}{\leq} \frac{\eta}{1-\eta\beta}\left({G+\frac{\beta}{2}}\right)\Vert{\Y_1-\X^*}\Vert_F^2,
\end{align}
where (a) follows from the $\beta$-smoothness of $f(\cdot)$, and (b) follows from Lemma \ref{lem:aux} and recalling that under Assumption \ref{ass:gap}, $\X^*$ is the projection matrix onto the span of the top $k$ eigenvectors of $-\nabla{}f(\X^*)$ (see Lemma \ref{lem:opt}).

Note that for $\eta\in(0,\beta)$, 
$\frac{\eta}{1-\eta\beta}\left({G+\frac{\beta}{2}}\right) \leq 1 ~ \Longleftrightarrow ~  \eta\left(G + \frac{3\beta}{2}\right) \leq 1$,
which clearly holds for our choice of step-size $\eta = \frac{1}{5\max\{\beta,G\}}$.

Thus, noting that for our choice of step-size and  initialization assumption it holds that $\Vert{\Y_1 - \X^*}\Vert_F \leq 1$, we indeed have that  $\Vert{\X_2 -\X_1}\Vert_F \leq 1$. Also, combining the initialization condition for $\Y_1$ listed in the theorem, together with Lemma \ref{lem:rad}, immediately implies that $\rank(\X_2) =k$. Thus, the base case $t=1$ of the induction holds.
 
Suppose now the induction holds for all $i\in\{1,\dots,t-1\}$, for some $t\geq 2$, and we will prove it for $t$. 
Using Lemma \ref{lem:distXY} we have that,
\begin{align}\label{eq:qr_thm:2}
\Vert{\X_{t+1}-\Y_t}\Vert_F^2 &\leq \frac{\eta}{1-\eta\beta}\left({f(\Y_{t})-f(\X_{t+1})}\right) \leq \frac{\eta}{1-\eta\beta}\left({f(\Y_{t})-f(\X^*)}\right) \nonumber \\
&\underset{(a)}{\leq} \frac{\eta}{1-\eta\beta}\left({f(\Y_{1})-f(\X^*)}\right) \underset{(b)}{\leq} \frac{\eta}{1-\eta\beta}\left({G+\frac{\beta}{2}}\right)\Vert{\Y_1-\X^*}\Vert_F^2,
\end{align}
where (a) follows by using the induction hypothesis for all $i\leq t$ together with  Lemma \ref{lem:errorDec}, which guarantees that $f(\Y_t) \leq f(\Y_1)$,  and (b) follows from the same steps as in \eqref{eq:qr_thm:1}.

Since we have already established that the RHS of \eqref{eq:qr_thm:2} is upper-bounded by $1$ in the base case of the induction, it follows that $\Vert{\X_{t+1}-\Y_t}\Vert_F \leq 1$.

Using the quadratic growth of $f(\cdot)$ (Lemma \ref{lem:quad}) we have that,
\begin{align*}
\Vert{\Y_t-\X^*}\Vert_F^2 &\leq \frac{2}{\delta}\left({f(\Y_t) - f(\X^*)}\right) \underset{(a)}{\leq} \frac{2}{\delta}\left({f(\Y_1) - f(\X^*)}\right)  \\
&\underset{(b)}{\leq} \frac{2}{\delta}\frac{\eta}{1-\eta\beta}\left({G+\frac{\beta}{2}}\right) \Vert{\Y_1-\X^*}\Vert_F^2,
\end{align*}
where again, (a) follows by using the induction hypothesis for all $i\leq t$ together with  Lemma \ref{lem:errorDec} which guarantees that $f(\Y_t) \leq f(\Y_1)$, and (b) follows from \eqref{eq:qr_thm:1}.

Thus, using the fact that for our choice of step-size it holds that $\frac{\eta}{1-\eta\beta}\left({G+\frac{\beta}{2}}\right) \leq 1$, using the initalization assumption on $\Y_1$, and invoking Lemma \ref{lem:rad}, it follows that indeed $\rank(\X_{t+1})=k$, i.e., $\X_{t+1}\in\mP_{n,k}$, and thus the induction holds for step $t$ as well.
\end{proof}

\section{Frank-Wolfe Analysis}
In this section we  prove Theorem \ref{fw_theorem}. Our analysis extends the one in \cite{garber2019linear} which only considered the case $k=1$.

We begin with a lemma, whose proof is similar to the arguments used in the proof
of Lemma \ref{lem:quad}, which will be essential to proving the local linear convergence of Frank-Wolfe under Assumption \ref{ass:gap}.

\bigskip \begin{lemma} \label{lem:miniquad}
Let $\X\in \mF_{n,k}$ and assume that $\lambda_{n-k}(\nabla f(\X))-\lambda_{n-k+1}(\nabla f(\X)) \geq \delta_{\X} > 0$. Then, for $\V\in\argmin_{\P\in\mP_{n,k}}\langle{\P,\nabla{}f(\X)}\rangle$ it holds that, 
\begin{align*}
\langle \X-\V,\nabla f(\X) \rangle \geq \dfrac{\delta_\X}{2}\| \X-\V\| _{F}^{2}.
\end{align*}
\end{lemma}

\begin{proof}
The proof follows the same lines as the proof of Lemma \ref{lem:quad}, but replacing $\X^*$ with $\X$, and noting that $\V$ is the (unique) projection matrix onto the span of the top $k$ eigenvectors of $-\nabla{}f(\X)$, similarly to the use of $\X^*$ as the projection matrix onto the span of the top $k$ eigenvectors of $-\nabla{}f(\X^*)$ in the proof of Lemma \ref{lem:quad}.
\end{proof}

\begin{algorithm}
\caption{Frank-Wolfe with line-search  for the Fantope}\label{alg:FW}
\begin{algorithmic}[1]
\State $\X_1 \gets$ arbitrary point in $\mF_{n,k}$
\For{$t=1... $}
      \State $\V_t \leftarrow \argmin_{\V\in\mP_{n,k}}\langle{\V,\nabla{}f(\X_t)}\rangle$
      \State Choose step size $\eta_t \in [0,1]$ using one of the two options:

      \State First Option : $\eta_t \leftarrow \argmin_{\eta\in[0,1]}f((1-\eta)\X_t+\eta\V_t)$

      \State Second Option : $\eta_t \leftarrow \argmin_{\eta\in[0,1]} f(\X_{t})+\eta\langle \V_{t}-\X_{t}, \nabla f(\X_{t}) \rangle+\frac{\eta^{2}\beta}{2}\|\V_{t}-\X_{t}\|_{F}^{2}$

      \State $\X_{t+1} \leftarrow (1-\eta_t)\X_t+\eta_t\V_t$
\EndFor
\end{algorithmic}
\end{algorithm}

\begin{theorem}[Formal version of Theorem \ref{fw_theorem}]
Let $\{\X_t\}_{t\geq1}$ be a sequence produced by Algorithm \ref{alg:FW} and denote $\forall t\geq1$, $h_t :=f(\X_t)-f(\X^*)$. Then, 
\begin{align} \label{fw:eq:1}
\forall t\geq1 : \quad h_t=O(k\beta/t).
\end{align}
In addition, if Assumption \ref{ass:gap} holds with parameter $\delta > 0$, then there exists $T_0=O\left({k(\beta/\delta)^3}\right)$ such that, 
\begin{align}\label{fw:eq:2}
\forall t\geq T : \quad h_{t+1}\leq h_{t}\left({1-\min\{ \dfrac{\delta}{12\beta}, \dfrac{1}{2} \}}\right)).
\end{align}
Finally, under Assumption \ref{ass:gap}, we have that,
\begin{align}\label{fw:eq:3}
\forall t\geq 1 : \quad \|\V_t-\X^*\|_F^2=O\left({\dfrac{\beta^2}{\delta^3}h_t}\right).
\end{align}
\end{theorem}

\begin{proof}
Result \eqref{fw:eq:1} follows from standard convergence results for the Frank-Wolfe method with line-search \cite{jaggi2013revisiting}, and the fact that the Euclidean diameter of the Fantope $\mF_{n,k}$ is $\sqrt{2k}$.

For the second part, observe that under Assumption \ref{ass:gap}, using the the $\beta$-smoothness of $f(\cdot)$, the quadratic growth result (Lemma \ref{lem:quad}) and \eqref{fw:eq:1}, we have that for all $t\geq 1$, 
\begin{align*}
\| \nabla f(\X_t) - \nabla f(\X^*) \|_{F} \leq \beta \|\X_t-\X^*\| \leq \beta \sqrt{\dfrac{2h_t}{\delta}} = O\left({\sqrt{\dfrac{k\beta^3}{t\delta}}}\right).
\end{align*}
Thus, for some $T_0=O\left({k(\beta/\delta)^3}\right)$ we have that,
\begin{align*}
\forall t\geq T_0: \quad \| \nabla f(\X_t) - \nabla f(\X^*) \|_{F} \leq \dfrac{\delta}{3}.
\end{align*}
Let us write the eigen-decomposition of $\nabla f(\X_t)$ as $\nabla f(\X_t)=\sum_{i=1}^{n} \lambda_i \v_i \v_i^{\top}$.
Using Weyl's inequality for the eigenvalues we can write for every $t\geq T_0$,
\begin{flalign*}
\lambda_{n-k}-\lambda_{n-k+1} &\geq \lambda_{n-k}(\nabla f(\X^*))-\lambda_{n-k+1}(\nabla f(\X^*)) -2\Vert{\nabla{}f(\X) - \nabla{}f(\X^*)}\Vert_F \nonumber\\
& \geq \delta - \dfrac{2\delta}{3} = \dfrac{\delta}{3}.
\end{flalign*}
Thus, for all $t\geq T_0$, $\lambda_{n-k+1}<\lambda_{n-k}$ and the matrix $\V_t$ is uniquely defined and given by $\V_t = \sum_{i=n-k+1}^n\v_i\v_i^{\top}$.  
Using $\X_{t+1}=(1-\eta_{t})\X_{t} + \eta_{t}\V_{t}$, the smoothness of $f(\cdot)$, and the fact that $\eta_t$ is chosen via line-search, we have that,
\begin{align*}
\forall \eta\in[0,1] : f(\X_{t+1}) \leq f(\X_t)+\eta\langle \V_{t}-\X_t, \nabla f(\X_t) \rangle + \dfrac{\eta^2 \beta}{2} \| \V_{t}-\X_t \|_F^2.
\end{align*}
Subtracting $f(\X^*)$ from both sides and using Lemma \ref{lem:miniquad} with gap $\delta_\X=\delta/3$, we have that for all $t\geq T_0$,
\begin{align*}
\forall \eta\in[0,1]:\quad h_{t+1}&\leq h_t+\dfrac{\eta}{2} \langle \V_{t} -\X_t, \nabla f(\X_t)\rangle +( \dfrac{\eta^2 \beta}{2} - \dfrac{\eta \delta}{12})\| \V_{t}-\X_t \|_F^2 \\
&\leq (1-\dfrac{\eta}{2})h_t+( \dfrac{\eta^2 \beta}{2}- \dfrac{\eta \delta}{12})\| \V_{t}-\X_t \|_F^2,
\end{align*}
where the last inequality follows from the convexity of $f(\cdot)$.

Now, if $\dfrac{\delta}{6\beta} \leq 1$, by setting $\eta=\dfrac{\delta}{6\beta}$ we have that $h_{t+1}\leq (1-\dfrac{\delta}{12\beta})h_{t}$.
Otherwise, $\delta>6\beta$ and so, setting $\eta=1$, we get that $h_{t+1}\leq \dfrac{1}{2} h_t$, which proves Result \eqref{fw:eq:2}.

Finally, for the third part of the lemma, recalling that $\V_t$ and $\X^*$ are the projection matrices onto the span of the top $k$ eigenvectors of $-\nabla{}f(\X_t)$ and $-\nabla{}f(\X^*)$, respectively, using the well known Davis-Kahan sin$\theta$ theorem (see for instance \cite{yu2015useful}), we have that for all $t\geq 1$,
\begin{align*}
 \|\V_t-\X^*\|_F^2\leq \dfrac{8 \|\nabla f(\X_t)-\nabla f(\X^*)\|_F^2}{(\lambda_{n-k} (\nabla f(\X^*)) - \lambda_{n-k+1} (\nabla f(\X^*)))^2} \leq \dfrac{8\beta^2 \|\X_t-\X^*\|_F^2}{\delta^2} \leq \dfrac{16\beta^2 h_t}{\delta^3},
\end{align*}
where the last inequality follows from the quadratic growth result, Lemma \ref{lem:quad}. Thus, Result \eqref{fw:eq:3} follows.
\end{proof}

\end{document}